\documentclass[10pt,reqno]{amsart}

\usepackage[utf8]{inputenc}
\usepackage[english]{babel}
\usepackage[dvipsnames]{xcolor}
\usepackage{amssymb}
\usepackage{hyperref}
\usepackage{mathrsfs}
\usepackage{mathtools}

\usepackage{xparse}% http://ctan.org/pkg/xparse
\NewDocumentCommand{\ceil}{s O{} m}{%
  \IfBooleanTF{#1} % starred
    {\left\lceil#3\right\rceil} % \ceil*[..]{..}
    {#2\lceil#3#2\rceil} % \ceil[..]{..}
}

\usepackage{color}
\usepackage{enumerate}

\newtheorem{theorem}{Theorem}[section]
\newtheorem{lemma}[theorem]{Lemma}
\newtheorem{corollary}[theorem]{Corollary}
\newtheorem{proposition}[theorem]{Proposition}

\theoremstyle{remark}
\newtheorem{remark}[theorem]{\bf{Remark}}

\theoremstyle{definition}
\newtheorem{assumption}[theorem]{Assumption}
\newtheorem{example}[theorem]{Example}
\newtheorem{definition}[theorem]{Definition}

\newcommand\cbrk{\text{$]$\kern-.15em$]$}}
\newcommand\opar{\text{\,\raise.2ex\hbox{${\scriptstyle
|}$}\kern-.34em$($}}
\newcommand\cpar{\text{$)$\kern-.34em\raise.2ex\hbox{${\scriptstyle |}$}}\,}

\newcommand{\aint}{-\hspace{-0.38cm}\int}
\newcommand{\aaint}{-\hspace{-0.31cm}\int}

\newcommand\bE{\mathbb{E}}

\newcommand\bH{\mathbb{H}}

\newcommand\bL{\mathbb{L}}

\newcommand\bN{\mathbb{N}}

\newcommand\bP{\mathbb{P}}
\newcommand\bQ{\mathbb{Q}}
\newcommand\bR{\mathbb{R}}

\newcommand\cF{\mathcal{F}}

\newcommand\cM{\mathcal{M}}

\newcommand\cO{\mathcal{O}}

\newcommand\cS{\mathcal{S}}

\newcommand\nn{\nonumber}

\newcommand{\mysection}[1]{\section{#1}
\setcounter{equation}{0}}

\newcommand{\Ccinf}{C_{c}^{\infty}}
\newcommand{\R}{\mathbb{R}}

\begin{document}

\title[PDE with nonlocal operators having slowly varying kernels]
{An $L_{q}(L_{p})$-regularity theory for parabolic equations with integro-differential operators having low intensity kernels}

\author{Jaehoon Kang}
\address{Hankyong National University, 327 Jungang-ro, Anseong-si, Gyeonggi-do 17579, Republic of Korea}
\email{jaehoon.kang@hknu.ac.kr}

\author{Daehan Park}
\address{Kangwon National University, 1 Kangwondaehakgil, Chucheon-Si, Gangwon State, 24341, Republic of Korea.} 
\email{daehanpark@kangwon.ac.kr}

\thanks{The second author was supported by the 2024 Research Grant from Kangwon Nationa University (grant no. 202404630001).}



\begin{abstract}
In this article, we present the existence, uniqueness, and  regularity of solutions to parabolic equations with non-local operators
$$
\partial_{t}u(t,x) = \mathcal{L}^{a}u(t,x) + f(t,x), \quad t>0
$$
in $L_{q}(L_{p})$ spaces. Our spatial operator $\mathcal{L}^{a}$ is an integro-differential operator of the form
$$
\int_{\mathbb{R}^{d}} \left( u(x+y)-u(x) -\nabla u(x) \cdot y \mathrm{1}_{|y|\leq 1}  \right) a(t,y) j_{d}(|y|)dy.
$$
Here, $a(t,y)$ is a merely bounded measurable coefficient, and we employed the theory of additive process to handle it. We investigate conditions on $j_{d}(r)$ which yield $L_{q}(L_{p})$-regularity of solutions. Our assumptions on $j_d$ are general so that $j_d(r)$ may be comparable to $r^{-d}\ell(r^{-1})$ for a function $\ell$ which is slowly varying at infinity. For example, we can take $\ell(r)=\log{(1+r^{\alpha})}$ or $\ell(r) = \min{\{r^{\alpha},1\}}$ ($\alpha\in(0,2)$). Indeed, our result covers the operators whose Fourier multiplier $\psi(\xi)$ does not have any scaling condition for $|\xi|\geq 1$. Furthermore, we give some examples of operators, which cannot be covered by previous results where smoothness or scaling conditions on $\psi$ are considered.
\end{abstract}

\maketitle

\tableofcontents

\mysection{Introduction}
Parabolic equations are among the most fundamental partial differential equations, which play a central role in the mathematical description of natural or artificial phenomena. Especially, parabolic equations are related to the diffusion of particles in various situations. There is another mathematical approach to the description of diffusion based on the theory of stochastic processes. If we let $B_{t}$ be a Brownian motion, then it turns out that its transition density is the fundamental solution to the heat equation, and the Laplacian $\Delta$ is the infinitesimal generator of $B_{t}$;
$$
\Delta u(x) = \lim_{t\downarrow 0} \frac{\mathbb{E}u(x+B_{t})-u(x)}{t}.
$$

The above observation still holds if we change $B_{t}$ to L\'evy processes. Therefore, the analysis of the heat equations can be extended to that of parabolic equations with the infinitesimal generator of L\'evy processes. For example, if we change $B_{t}$ to symmetric $\alpha$-stable ($\alpha\in(0,2)$) process, then the corresponding equation is $\partial_{t}u = -(-\Delta)^{\alpha/2}u$, and $-(-\Delta)^{\alpha/2} = \Delta^{\alpha/2}$ (which we call it a fractional Laplacian of order $\alpha$) is the infinitesimal generator of symmetric $\alpha$-stable process. It is well-known that $\Delta^{\alpha/2}$ is a non-local operator with the Fourier multiplier $-|\xi|^{\alpha}$. In general, the infinitesimal generators of L\'evy processes are non-local operators since L\'evy processes can move by jumps (see e.g. \cite{sato1999}).

In this article, we study $L_{p}$-regularity theory of parabolic equations with non-local operators. The basic form of our spatial  non-local operator is
\begin{equation}\label{eqn 07.25.17:38}
\mathcal{L}u(x) = \int_{\mathbb{R}^{d}} \left( u(x+y)-u(x) -\nabla u(x) \cdot y \mathrm{1}_{|y|\leq 1}  \right) j_{d}(|y|)dy,
\end{equation}
and we also consider more general operators by including bounded measurable coefficients (see \eqref{eqn 03.15.14:09} for detail). We investigate conditions on $j_d$ which yield $L_{p}$-regularity of solutions. The assumptions are quite general, and they cover not only the fractional Laplacian whose (jumping) kernel is given as $j_{d}(r) = c_{d,\alpha} \, r^{-d-\alpha}$ but also the kernel $j_{d}$ which has low intensity near zero (for example, the kernel $j_{d}(r)$ which is comparable to $r^{-d}$ near zero). In this case, $j_{d}$ has no explicit scaling condition near zero.

The study of non-local operators having low intensity near zero has received growing interest recently. In \cite{CoPa18} the authors proved Hardy inequality and compact embedding related to non-local operators having low intensity near zero. In \cite{KassM17}, a H\"older estimation of solutions to equations with such types of operators are studied.  There are other types of non-local operators, whose characteristic functions are slowly varying function at infinity (i.e. $\psi(\lambda r)/\psi(r) \to 1$ as $r\to \infty$ for any $\lambda>0$). For example, the operator $L_{\Delta}:= \frac{d}{d\alpha}\big|_{\alpha=0}[(-\Delta)^{\alpha/2}]$ does not have the representation like \eqref{eqn 07.25.17:38} even though its Fourier multiplier is $2\log{|\xi|}$ (see \cite{CW19}). There are results for the theory of analysis and partial differential equations related to it in e.g. \cite{CW19,JSW20,LW21,ZN21}. We refer to \cite{DGV23} which addresses some applications of non-local operators having low intensity kernels (see Remark 2.8 and Remark 2.9 therein). We also refer to very recent results \cite{CK24log,HZ23nonlocal} considering non-local operators with low intensity kernels.

Let us review some previous results in the literature on $L_{p}$-theory. To the best of the authors' knowledge, $L_{p}$-theory of parabolic equations with non-local operators was first established in \cite{mik1992}. The operators therein are of the form 
\begin{equation*}
\mathscr{L}^{b}u(x)=\int_{\bR^{d}} \left( u(x+y) - u(x) -(\nabla u(x) \cdot y)\chi_{\alpha}(y) \right)  \frac{b(x,y)}{|y|^{d+\alpha}}dy,
\end{equation*}
where $\chi_{\alpha}(y) = \mathbf{1}_{\alpha=1}\mathbf{1}_{|y|\leq 1} + \mathbf{1}_{\alpha\in(1,2)}$, and $b(x,y)$ is uniformly continuous in $x$, differentiable in $y$ up to $\lfloor d/2 \rfloor +1$. There are results in the study of  $L_{p}$-theory for elliptic and parabolic equations which relax the condition on $b$ in e.g. \cite{DJK23nonlocal,DK12elliptic,MP14cauchy}. Note that $\mathscr{L}^{b}$ can be considered as fractional Laplacian with coefficients. Thus the above results correspond to the theory of parabolic equations with variable coefficients
$$
\partial_{t}u(t,x) = \sum_{i,j=1}^{d}a^{ij}(t,x)u_{x^{i}x^{j}}(t,x) + f(t,x),
$$
which has been studied for decades.

On the other hand, there are results, e.g. \cite{kim2013parabolic,mikulevivcius2019cauchy,mikulevivcius2017p,zhang2013p}, which focus on generalization of principal part  $|y|^{-d-\alpha}$  for $L_{p}$-theory of parabolic equations. One of the examples, which can be handled by the above results is $\mathcal{L} = -\phi(-\Delta)$, where $\phi$ is a Bernstein function with the following (weak) lower and upper scaling condition
\begin{equation}\label{eqn 07.25.17:28}
C_{1}\left( \frac{R}{r} \right)^{\delta} \leq \frac{\phi(R)}{\phi(r)} \leq C_{2} \left( \frac{R}{r} \right)^{\delta '} \quad \forall\, 0<r<R<\infty, \quad 0<\delta \leq \delta '<1.
\end{equation}
We also refer to very recent result \cite{JSS23} which adopts settings in \cite{mikulevivcius2019cauchy,mikulevivcius2017p}. For $L_{p}$-theory of parabolic equations with pseudo differential operators having smooth symbols, we refer e.g. \cite{CK23pseudo,kim2019L,KKL16}. 

For equations with space-time non-local operators, we refer e.g. \cite{clement1992global,P13,zacher2005maximal,Z12}, which deal with parabolic Volterra equations of the type
$$
u(t) + \int_{0}^{t} a(t-s) Au(s) ds = f(t),
$$
where $a$ is locally integrable function and $A$ is densely defined closed operator on $L_{p}$. See also e.g. \cite{KW04} for parabolic equations with a similar approach. For $L_{q}(L_{p})$-theory of such equations with $a(t)=t^{\alpha-1}/\Gamma(\alpha)$, and an integro-differential operator $A$ having low intensity kernel, see \cite{KP21}.

Our main result is parabolic correspondence of \cite{KP21} (i.e. $a(t,y) \equiv 1$) and we will consider $j_{d}$ satisfying the following condition;
There is a continuous function $\ell:(0,\infty) \to (0,\infty)$ such that $j_{d}(r)$ is comparable to $r^{-d}\ell(r^{-1})$ and $\ell$ satisfies
\begin{align}
C_{1}\left(\frac{R}{r}\right)^{\delta_{1}} \leq \frac{\ell(R)}{\ell(r)} \leq C_{2} \left( \frac{R}{r} \right)^{\delta_{2}} \quad \forall \, 1\leq r\leq R<\infty,\label{assinf}
\\
C_{1}\left(\frac{R}{r}\right)^{\delta_{3}} \leq \frac{\ell(R)}{\ell(r)} \leq C_{2}\left(\frac{R}{r}\right)^{\delta_{4}}\quad \forall \, 0<r\leq R\leq 1, \nonumber
\end{align}
where $0\leq \delta_{1} \leq \delta_{2}<2$ and $0<\delta_{3} \leq \delta_{4} <2$. For additional assumptions on $j_d$, see Assumption \ref{ell_con} and Definition \ref{asm 07.10.16:46}  below. Note that under the condition \eqref{eqn 07.25.17:28}, it turns out that $j_{d}(r)$ is comparable to $r^{-d}\phi(r^{-2})$ (see e.g. \cite{KSV14}), and hence $j_{d}$ satisfies our assumptions with $\ell(r) =\phi(r^{2})$. To study the case where $\delta$ in \eqref{eqn 07.25.17:28} is zero, the scaling condition on $\phi'$, the derivative of $\phi$, is considered in \cite{KM12} instead of the scaling condition on $\phi$. In \cite{KM12}, the authors proved the Harnack inequality for harmonic functions, and estimation of jumping kernels and Green funstions related to corresponding stochastic processes. Under the assumption in \cite{KM12}, we see that $j_{d}(r)$ is comparable to $r^{-d-2}\phi'(r^{-2})$ for $r<1$ and thus we can take $\ell(r)=r^2\phi'(r^2)$ to cover this case.

Now, we explain the advantages of our assumptions. Since we find conditions on $j_d$, our assumptions are more fundamental for the operators given by \eqref{eqn 07.25.17:38}. In the literature  (like \cite{KKL16,KW04}), the conditions (especially smoothness) on  $\psi$, the Fourier multiplier of $\mathcal{L}$, are considered to study non-local operators with general kernel in abstract spaces. Thus, under these approaches, we need to find properties of $\psi$. Non-local operators $-\phi(-\Delta)$ for Bernstein functions $\phi$ fit such approaches and thus they are widely considered (see e.g. \cite{kim2013parabolic,kim2020nonlocal,KPJ22}). However, to the best of the authors' knowledge, the only known relation between $\psi$ and $j_{d}$ is
$$
\psi(\xi) = \int_{\bR^{d}} \left( 1-\cos{(\xi \cdot y)} \right) j_{d}(y) dy.
$$
Thus, to derive properties of $\psi$, we eventually need to impose appropriate conditions on $j_{d}$. Moreover, it seems nontrivial to show desired properties of $\psi$ for general jumping kernel $j_{d}$. Hence, analyzing $\mathcal{L}$ directly under assumption on $j_{d}$ (this approach has been widely used for scalable jumping kernels in e.g. \cite{JSS23,KK16Elliptic,mikulevivcius2019cauchy,mikulevivcius2017p,zhang2013p}) is more essential when the operator is not given with the Fourier multiplier. Indeed, our approach makes it possible to handle operators beyond $-\phi(-\Delta)$, and we do not require any condition (especially, smoothness and scaling conditions) on $\psi$. In Theorem \ref{main theorem2} below, we impose 4-times (which seems optimal in our approach) differentiability on $j_{d}$ which makes it possible to handle operators beyond $-\phi(-\Delta)$.  See Remark \ref{rmk 08.31.14:17} for details of the number of differentiability.

The second advantage is that we allow the constant $\delta_1$ to be zero and thus $\ell$ can be comparable to a slowly varying function at infinity. Here, $f:(0,\infty) \to (0,\infty)$ is called a slowly varying function at infinity if it satisfies $\lim_{r\to\infty} f(\lambda r)/f(r) =1$ for all $\lambda>0$. Since we consider the operator $\mathcal{L}$ given by \eqref{eqn 07.25.17:38} with the kernel $j_d(r)$ which is comparable to $r^{-d}\ell(r^{-1})$, our result covers the kernel which has low singularity near zero. If $\ell$ is a slowly varying function at infinity, then the corresponding Fourier multiplier $\psi$ is comparable to a slowly varying function at infinity. Thus, in this case, $\psi$ does not satisfy the scaling condition. Note that a lower scaling condition on $\psi$ or $\ell$ with positive exponent is essential in the results mentioned above. See Example \ref{exm 10.05.14:52} for examples of $j_{d}$ which satisfies our assumptions.

We show $L_{p}$-regularity of solution by using Calder\'on-Zygmund approach and heat kernel estimates obtained by \cite{cho21heat, KP21}. Since the function $\ell$ can be a slowly varying function at infinity, the heat kernel has different types of estimates compared to the case where $\ell$ satisfies the scaling condition. Indeed, if $\ell(r)$ is comparable to a constant for $r>1$, then the corresponding heat kernel $p(t,x)$ satisfies that $\sup_{x\in\R^d}p(t,x)=\infty$ while there exists $C>0$  such that  $p(t,x)\le Ct^{-d/\alpha}$ for all $t>0$ and $x\in \R^d$ if $\ell(r)=r^{\alpha}$ for $\alpha\in(0,2)$. As we will see in Section \ref{sec3}, the heat kernel bounds are complicated compared to the one for the case of $\delta_1>0$ in \eqref{assinf}, and thus we need more exquisite calculations for BMO-$L_{\infty}$ estimation of the maximal derivative of solution (We describe the outline of our argument in Section \ref{sec 10.05.15:24});
$$
\|\mathcal{L}u\|_{BMO(\bR^{d+1})} \leq C \|f\|_{L_{\infty}(\bR^{d+1})}.
$$
If $\delta_1>0$, then we can check that 
$$
h(r)=r^{-2}\int_{0}^{r}s\ell(s^{-1})ds+\int_{r}^{\infty} s^{-1}\ell(s^{-1}) ds
$$
is comparable to $\ell(r^{-1})$, and the heat kernel $p(t,x)$ is comparable to  (see, for example, \cite[Remark 2.11]{cho21heat})
$$
(h^{-1}(t^{-1}))^{-d}\wedge \frac{t\ell(|x|^{-1})}{|x|^d},
$$
for all $t>0$ and $x\in\R^d$. Our assumptions admit the case where $h(r)$ is not comparable to $\ell(r^{-1})$. The delicate analysis arises here since the size of the parabolic cube is given by using $h$, whereas both $h$ and $\ell$ appear in the estimation of the heat kernel. Moreover, the lack of the scaling condition yields heat kernel with a more complicated form for analysis (see Section \ref{sec3} for detail).  

We finish the introduction with notations. We use $``:="$ or $``=:"$ to denote a definition. The symbol $\bN$ denotes the set of positive integers and $\bN_0:=\bN\cup\{0\}$. As usual $\bR^d$ stands for the Euclidean space of points $x=(x^1,\dots,x^d)$. We set
$$
B_r(x):=\{y\in \bR: |x-y|<r\}, \quad \bR_+^{d+1} := \{(t,x)\in\bR^{d+1} : t>0 \}.
$$
For $i=1,\ldots,d$,
multi-indices $\sigma=(\sigma_{1},\ldots,\sigma_{d})$,
and functions $u(t,x)$ we set
$$
\partial_{x^{i}}u=\frac{\partial u}{\partial x^{i}}=D_{i}u,\quad D^{\sigma}u=D_{1}^{\sigma_1}\cdots D_{d}^{\sigma_d}u,\quad|\sigma|=\sigma_{1}+\cdots+\sigma_{d}.
$$
We also use the notation $D_{x}^{m}$ for arbitrary partial derivatives of
order $m$ with respect to $x$.
For an open set $\cO$ in $\bR^{d}$ or $\bR^{d+1}$, $C_c^\infty(\cO)$ denotes the set of infinitely differentiable functions with compact support in $\cO$. By
$\cS=\cS(\bR^d)$ we denote the class of Schwartz functions on $\bR^d$.
For $p> 1$, by $L_{p}$ we denote the set
of complex-valued Lebesgue measurable functions $u$ on $\R^{d}$ satisfying
\[
\left\Vert u\right\Vert _{L_{p}}:=\left(\int_{\R^{d}}|u(x)|^{p}dx\right)^{1/p}<\infty.
\]
Generally, for a given measure space $(X,\mathcal{M},\mu)$, $L_{p}(X,\cM,\mu;F)$
denotes the space of all $F$-valued $\mathcal{M}^{\mu}$-measurable functions
$u$ so that
\[
\left\Vert u\right\Vert _{L_{p}(X,\cM,\mu;F)}:=\left(\int_{X}\left\Vert u(x)\right\Vert _{F}^{p}\mu(dx)\right)^{1/p}<\infty,
\]
where $\mathcal{M}^{\mu}$ denotes the completion of $\cM$ with respect to the measure $\mu$.
If there is no confusion for the given measure and $\sigma$-algebra, we usually omit the measure and the $\sigma$-algebra.
We denote $a\wedge b := \min\{a,b\}$ and $a\vee b:=\max\{a,b\}$. By $\cF_{d}$ and $\cF^{-1}_{d}$ we denote the $d$-dimensional Fourier transform and the inverse Fourier transform respectively, i.e.
$$
\cF_{d}(f)(\xi):=\hat{f}(\xi):=\int_{\bR^d} e^{-i\xi\cdot x} f(x)dx, \quad \cF^{-1}_{d}(f)(\xi):=\frac{1}{(2\pi)^d}\int_{\bR^d} e^{i\xi\cdot x} f(x)dx.
$$
For any $a,b>0$, we write $a\asymp b$ if there is a constant $c>1$ independent of $a,b$ such that $c^{-1}a\leq b\leq ca$. Finally, if we write $C=C(\dots)$, this means that the constant $C$ depends only on what are in the parentheses. The constant $C$ can differ from line to line.

\medskip

\mysection{Main results} \label{Main result sec}
In this section, we introduce our main results. We first present our spatial non-local operator $\mathcal{L}$. 
For $f\in \cS(\bR^{d})$, define a linear operator $\mathcal{L}$ as
$$
\mathcal{L}f(x) = \int_{\bR^{d}} \left( f(x+y)-f(x)-\nabla f(x) \cdot y \mathbf{1}_{|y| \leq 1} \right) j_{d}(|y|)dy,
$$
where $r\mapsto j_d(r)$ is decreasing and satisfies 
\begin{align}\label{levym}
\int_{\R^d}(1\wedge |x|^2)j_d(|x|)dx<\infty.
\end{align}
Now we introduce objects related to the non-local operator $\mathcal{L}$. Since $j_d(r)$ is decreasing in $r$ and satisfies \eqref{levym}, there is a pure jump isotropic unimodal (see the next paragraph for terminologies) L\'evy process $X=(X_t, t\ge0)$  on $\bR^d$, whose characteristic exponent $\psi$ is given by
\begin{align*}
\psi(\xi)=\psi_X(\xi)=\int_{\R^d}(1-\cos(\xi \cdot x))j_d(|x|)dx
\end{align*}
(see \cite{sato1999,W83}). 

Note that a measure on $\R^d$ is called isotropic unimodal if it is absolutely continuous on $\R^d\setminus\{0\}$, with respect to the Lebesgue measure, with a radial and radially decreasing density. A L\'evy process $X=(X_t, t\ge0)$  on $\bR^d$ is isotropic unimodal if $p_d(t, dx):= \bP(X_t\in dx)$ is isotropic unimodal for all $t > 0$. We also note that the characteristic (or L\'evy-Khintchine) exponent $\psi$ of $X$ is defined by
\begin{align*}
\bE e^{i \xi \cdot X_t} =\int_{\R^d}e^{i \xi \cdot x}p_d(t,dx)=e^{-t\psi(\xi)}.
\end{align*}
Under the above setting, we call $j_d(x)=j_d(|x|)$ the jumping kernel for $X$. Also, in this situation, we allow the following abuse of notations;
$$\psi(|x|) = \psi(x), \quad j_d(x)dx=j_d(|x|)dx=j_d(dx) \quad \forall\, x \in \R^d.$$

By \cite[Theorem 31.5]{sato1999} we can understand $\mathcal{L}$ as the infinitesimal generator of $X$, and nonlocal operator with Fourier multiplier $-\psi(|\xi|)$. Precisely speaking, for $f\in \cS(\bR^{d})$,  we have the following relation
\begin{equation}\label{eqn 06.28.11:33}
\mathcal{L}f(x) = \lim_{t \downarrow 0} \frac{ \bE f(x+X_t) -f(x)}{t} = \cF^{-1}(-\psi(|\xi|)\cF(f)(\xi))(x).
\end{equation}
In this context, we also use notations $\mathcal{L}_{\psi}$ or $\mathcal{L}_{X}$ instead of $\mathcal{L}$ in this article. Also, it is known that (see e.g. \cite[Lemma A.1]{KPJ22}) for $f\in \mathcal{S}(\bR^{d})$ and $r>0$, 
\begin{align}\label{eqn 09.28.11:26}
\mathcal{L}f(x) = \int_{\bR^{d}} \left( f(x+y)-f(x)-\nabla f(x) \cdot y \mathbf{1}_{|y| \leq r} \right) j_{d}(y)dy.
\end{align}

Now we give an important example of isotropic unimodal L\'evy processes
\begin{example}\label{exm 04.04.17:27}
Let $B=(B_{t},t\geq0)$ be  a $d$-dimensional Brownian motion and $S=(S_{t}, t\geq0)$ be a subordinator (i.e., $1$-dimensional increasing L\'evy process) independent of $B$. A stochastic process $Y=(Y_{t},t\geq0)$  defined by $Y_t:=B_{S_{t}}$ is called subordinate Brownian motion. Then there is a function $\phi$ with the following representation 
\begin{equation*}
\phi(\lambda)=b\lambda + \int_{(0,\infty)} (1-e^{-\lambda t})\mu(dt),
\end{equation*}
satisfying $\bE[e^{-\lambda S_t}]=e^{-t\phi(\lambda)}$. 
Here, $b\geq 0$ and $\mu$ is a measure  satisfying  $\int_{(0,\infty)} (1\wedge t) \mu(dt)<\infty$. We call $\phi$ a Bernstein function and  $\mu$ the L\'evy measure of $\phi$ (see \cite{SSV12}).

It is well-known that the L\'evy-Khintchine exponent of $Y$ is $\xi\mapsto\phi(|\xi|^2)$ and the L\'evy measure of $Y$ has the density $J_d(x)=J_d(|x|)$, where
\begin{equation}\label{eqn 09.07.17:36}
J_{d}(r)=\int_{(0,\infty)} (4\pi t)^{-d/2} e^{-r^2/(4t)} \mu(dt).
\end{equation}
Hence, $\mathcal{L}_{Y}$ also has representation \eqref{eqn 06.28.11:33} and \eqref{eqn 09.28.11:26} with $\phi(|\cdot|^{2})$ and $J_{d}$ in place of $\psi$ and $j_{d}$  respectively. In particular, by taking $\phi(\lambda)=\lambda^{\alpha/2}$ ($\alpha\in (0,2)$), we obtain the fractional Laplacian $\Delta^{\alpha/2}=-(-\Delta)^{\alpha/2}$, which is the infinitesimal generator of a rotationally symmetric $\alpha$-stable process in $\bR^d$. 
\end{example}

To describe the regularity of solutions, we introduce Sobolev space related to the operator $\mathcal{L}$.
For $\gamma\in\bR$, and $u\in \cS(\bR^{d})$, define linear operators 
\begin{equation*}
(-\mathcal{L})^{\gamma/2}=(-\mathcal{L}_{\psi})^{\gamma/2}, \quad (1-\mathcal{L})^{\gamma/2}=(1-\mathcal{L}_{\psi})^{\gamma/2}
\end{equation*}
as follows
\begin{align*}
\cF\{(-\mathcal{L})^{\gamma/2}u\} = (\psi(|\xi|))^{\gamma/2}\cF(u)(\xi), \quad
\cF\{(1-\mathcal{L})^{\gamma/2}u\} = (1+\psi(|\xi|))^{\gamma/2}\cF(u)(\xi).
\end{align*}
For $1<p<\infty$, let $H_p^{\psi,\gamma}$ be the closure of $\cS(\bR^{d})$ under the norm
\begin{equation*}
\|u\|_{H_p^{\psi,\gamma}}:=\|\cF^{-1}\{\left(1+\psi(|\cdot|)\right)^{\gamma/2}\cF(u)(\cdot)\}\|_{L_p}<\infty.
\end{equation*}
Then from the definition of $H^{\psi,\gamma}_{p}$ the operator $(1-\mathcal{L})^{\gamma/2}$ can be extended from $\cS(\bR^{d})$ to $L_{p}$. Throughout this article, we use the same notation $(1-\mathcal{L})^{\gamma/2}$ for this extension. For more information, see e.g. \cite{farkaspsi}.  Also note that if $\psi(|\xi|)=|\xi|^{2}$, then $H^{\psi,\gamma}_{p}$ is a standard Bessel potential space $H^{\gamma}_{p}$ and $H^{\psi,0}_{p}=L_{p}$ due to the definition.

The following lemma is a collection of useful properties of $H^{\psi,\gamma}_{p}$. For the proof, see e.g. \cite[Lemma 2.1]{KP21}.
\begin{lemma}\label{H_p^phi,gamma space}
Let $1<p<\infty$ and let $\gamma\in\bR$.

(i) The space $H_p^{\psi,\gamma}$ is a Banach space.

(ii) For any $\mu\in\bR$, the map $(1-\mathcal{L})^{\mu/2}$ is an isometry from $H^{\psi,\gamma}_{p}$ to $H^{\psi,\gamma-\mu}_{p}$.

(iii) If $\mu>0$, then we have continuous embeddings $H_p^{\psi,\gamma+\mu}\subset H_p^{\psi,\gamma}$ in the sense that
\begin{equation*}
\|u\|_{H_p^{\psi,\gamma}}\leq C \|u\|_{H_p^{\psi,\gamma+\mu}},
\end{equation*}
where the constant $C$ is independent of $u$. 

(iv) For any $u\in H^{\psi,\gamma+2}_{p}$, we have
\begin{equation}\label{eqn 03.25.15:03}
\left(\|u\|_{H^{\psi,\gamma}_p}+\|\mathcal{L}u\|_{H^{\psi,\gamma}_p}\right) \asymp  \|u\|_{H_p^{\psi,\gamma+2}}.
\end{equation}

\end{lemma}

For $p,q\in(1,\infty), \gamma\in\bR$ and $T<\infty$, we denote
\begin{equation*}
\bH_{q,p}^{\psi,\gamma}(T):=L_q\left((0,T); H_p^{\psi,\gamma}\right), \qquad \bL_{q,p}(T):=\bH_{q,p}^{\psi,0}(T).
\end{equation*}
We write $u\in C_{p}^{\infty}([0,T]\times\bR^d)$ if $D^m_x u, \partial_t D^m_x u \in C([0,T];L_p)$ for any $m\in \bN_{0}$. 

\begin{definition}
 \label{defn defining}
Let  $1<p,q<\infty$, $\gamma\in\bR$, and $T<\infty$.

(i) We write $u\in {\bH_{q,p}^{\psi,\gamma+2}}(T)$ if there exists a sequence $u_n\in C^{\infty}_{p}([0,T]\times \bR^d)$ satisfying
\begin{equation*}
\|u-u_n\|_{\bH_{q,p}^{\psi,\gamma+2}(T)} \to 0 \quad \text{and} \quad \|\partial_{t}u_{n} - \partial_{t}u_{m} \|_{\mathbb{H}^{\psi,\gamma}_{q,p}(T)} \to 0
\end{equation*}
as $n,m \to \infty$. We call this sequence $u_n$ a defining sequence of $u$, and we define
\begin{equation*}
\partial_t  u= \lim_{n \to \infty} \partial_t u_n \text{ in } \bH_{q,p}^{\psi,\gamma}(T).
\end{equation*}
The norm in $\bH_{q,p}^{\psi,\gamma+2}(T)$ is naturally given by
\begin{equation*}
\|u\|_{{\bH_{q,p}^{\psi,\gamma+2}}(T)}=\|u\|_{\bH_{q,p}^{\psi,\gamma+2}(T)}+\|\partial_{t}u\|_{\bH_{q,p}^{\psi,\gamma}(T)}.
\end{equation*}

(ii) We write $u\in{\bH_{q,p,0}^{\psi,\gamma+2}(T)}$, if there is a defining sequence $u_{n}$ of $u$ such that $u_{n}(0,\cdot)=0$ for all $n$.

\end{definition}

\begin{remark} \label{Hvaluedconti}
(i) Obviously, $\bH_{q,p}^{\psi,\gamma+2}(T)$ is a Banach space.

(ii) By following the argument in \cite[Remark 3]{mikulevivcius2017p}, we can show that the embedding $H_p^{2n} \subset H_p^{\psi,2n}$ is continuous for any $n\in\bN$ (see e.g. \cite[Remark 2.3 (ii)]{KP21}). 
 \end{remark}

Now we give some basic properties of $\mathbb{H}^{\psi,\gamma+2}_{q,p}(T)$. For proof, see Section \ref{s:App}.

\begin{lemma} \label{basicproperty}
Let  $1<p,q<\infty$, $\gamma\in\R$, and $T<\infty$.

(i) The space $\bH_{q,p,0}^{\psi,\gamma+2}(T)$ is a closed subspace of $\bH_{q,p}^{\psi,\gamma+2}(T)$.

(ii) $C_c^\infty(\bR^{d+1}_+)$ is dense in $\bH_{q,p,0}^{\psi,\gamma+2}(T)$.

(iii) For any $\gamma,\nu\in\bR$, $(1-\mathcal{L})^{\nu/2}:\bH_{q,p}^{\psi,\gamma+2}(T)\to\bH_{q,p}^{\psi,\gamma-\nu+2}(T)$ is an isometry and for any $u\in \mathbb{H}^{\psi,\gamma+2}_{q,p}(T)$, we have
$$
\partial_{t}(1-\mathcal{L})^{\nu/2}u = (1-\mathcal{L})^{\nu/2}\partial_{t}u. 
$$
\end{lemma}

Now, we present a collection of our assumptions and related facts.

\begin{assumption}\label{ass bernstein}
Let $\ell:(0,\infty) \to (0,\infty)$ be a continuous function satisfying
\begin{align}\label{H:s}
&C_{1}\left(\frac{R}{r}\right)^{\delta_{1}} \leq \frac{\ell(R)}{\ell(r)} \leq C_{2} \left( \frac{R}{r} \right)^{\delta_{2}} \quad \forall \, 1\leq r\leq R<\infty,
\end{align}
where the constants $C_{1},C_{2}>0$ and $0\leq \delta_{1}\leq \delta_{2}<2$. Also, we assume that there exists $0<\delta_{3}\leq \delta_{4}<2$ such that
\begin{align}\label{H:l}
C_{1}\left(\frac{R}{r}\right)^{\delta_{3}} \leq \frac{\ell(R)}{\ell(r)} \leq C_{2}\left(\frac{R}{r}\right)^{\delta_{4}}\quad \forall \, 0<r\leq R\leq 1.
\end{align}
\end{assumption}

Now, define 
\begin{equation*}
\begin{gathered}
K(r):=r^{-2}\int_{0}^{r}s\ell(s^{-1})ds, \quad L(r):=\int_{r}^{\infty} s^{-1}\ell(s^{-1}) ds,\quad
h(r):= K(r) + L(r).
\end{gathered}
\end{equation*}

\begin{remark}
Since $h'(r)=K'(r)+L'(r) = -2r^{-1}K(r)<0$, $h$ is strictly decreasing in $r$. Thus, the inverse function $h^{-1}$ of $h$ is well-defined. 
\end{remark}

Now we introduce the second assumption on $\ell$. Depending on whether $\ell$ is bounded or not, we have two different types of heat kernel bounds.

\begin{assumption}\label{ell_con}
The function $\ell$ in Assumption \ref{ass bernstein} satisfies

\textit{either}
\\
(i)  $\limsup_{r\to\infty}\ell(r)<\infty$;

\textit{or}
\\
(ii)  $ \limsup_{r\to\infty}\ell(r)=\infty \quad\text{and}\quad\ell(r)\asymp\sup_{s\le r}\ell(s)$. 

If $\ell$ satisfies (ii), then $\ell$ further satisfies that

\hspace{5mm} \textit{either} (ii)--(1) 
\begin{equation*}
h(r) \asymp \ell(r^{-1}) \quad \text{for}\; \, r \leq 1
\end{equation*}

\hspace{5mm} \textit{or} (ii)--(2)  for any $a>0$ there is a constant $C(a)>0$ such that
\begin{equation}\label{eqn 08.12.17:59}
\sup_{0<r<1} h(r)\exp{\left( -a\frac{h(r)}{\ell(r^{-1})}  \right)} \leq C(a).
\end{equation}

\end{assumption}

Note that $K$, $L$ and $h$ are independent of the dimension $d$.
Clearly, $L$ is strictly decreasing in $r$. Due to the definition, we easily see that   
\begin{equation*}
\begin{gathered}
 K(r) \asymp r^{-2}\int_{|y|\leq r}|y|^{2} j_{d}(|y|)dy, 
\quad L(r) \asymp \int_{|y|\geq r} j_{d}(|y|)dy, 
\quad
h(r) \asymp r^{-2}\int_{\bR^{d}} \left( r^{2} \wedge |y|^{2} \right) j_{d}(|y|)dy \quad \forall\,r>0.
\end{gathered}
\end{equation*}

For $n\in \bN_{0}$, and a function $f:(0,\infty) \to (0,\infty)$, define
$$
\mathcal{T}^{0}f(r) = f(r), \quad \mathcal{T}f(r)=:-\frac{1}{r}\left(\frac{d}{dr}f\right)(r), \quad \mathcal{T}^{n}f(r) :=\mathcal{T}^{n-1}(\mathcal{T}f)(r) \quad n\geq 2.
$$

The following condition will be used for our jumping kernel $j_{d}$.

\begin{definition}\label{asm 07.10.16:46} 
Let $d\in \bN$, and let $m\in\bN_{0}$.
We say that a function $f:(0,\infty)\to(0,\infty)$ satisfies $\bold{H}(d,m)$ if the following holds:

(i) $f$ and $\mathcal{T}f ,\dots, \mathcal{T}^{m}f$
are nonnegative and decreasing in $r\in(0,\infty)$.

(ii) For each $n \leq m$, There exist constants $\kappa_{1,n},\kappa_{2,n}>0$ such that for $r>0$,
\begin{equation}\label{e:H}
\kappa_{1,n} r^{-d-n} \ell(r^{-1}) \leq (-1)^{n}\frac{d^{n}}{dr^{n}} f(r) \leq \kappa_{2,n} r^{-d-n}\ell(r^{-1}).
\end{equation}
\end{definition}

In the rest of the article, we use notations 
\begin{gather*}
\text{{\boldmath$\delta$}} =(\delta_{1},\delta_{2},\delta_{3},\delta_{4}), \quad
\text{{\boldmath$\kappa$}}_{m}= (\kappa_{1,0},\kappa_{2,0},\dots,\kappa_{1,m},\kappa_{2,m}),  \quad 
\text{{\boldmath$\kappa$}}= (\kappa_{1,0},\kappa_{2,0},\dots)
\end{gather*}
instead of listing $\delta_{i},\kappa_{i,j}$ for notational convenience.

\begin{remark}\label{rmk 09.27.14:11}
(i) If $J_{d}$ is the jump kernel of subordinate Brownian motion, then using \eqref{eqn 09.07.17:36}, we can check that for any $n\in \bN$, $\mathcal{T}^{n}J_{d}$ is nonnegative and decreasing function.

(ii) Using \eqref{e:H}, if $j_{d}$ satisfies $\bold{H}(d,m)$, then we can check that  $\mathcal{T}j_{d}(r) \asymp r^{-d-2}\ell(r^{-1})$. Also, by the product rule of differentiation, for any $n\leq m-1$, we have 
\begin{align*}
(-1)^{n}\frac{d^{n}}{dr^{n}}\left( \mathcal{T}j_{d} \right)(r)&=(-1)^{n}\frac{d^{n}}{dr^{n}}\left( -\frac{1}{r}\frac{d}{dr}j_{d} \right)(r)
\\
&=(-1)^{n+1}\sum_{k=0}^{n} \binom{n}{k}(-1)^{k}k!r^{-1-k}\frac{d^{n-k+1}}{dr^{n-k+1}}j_{d}(r)
\\
&=\sum_{k=0}^{n} \binom{m}{k}(-1)^{k}k!r^{-1-k}(-1)^{k}(-1)^{n-k+1}\frac{d^{n-k+1}}{dr^{n-k+1}}j_{d}(r)
\\
&=\sum_{k=0}^{n} \binom{n}{k}k!r^{-1-k}(-1)^{n-k+1}\frac{d^{n-k+1}}{dr^{n-k+1}}j_{d}(r).
\end{align*}
Hence, $\mathcal{T}j_{d}$ satisfies $\bold{H}(d+2,m-1)$. We can also check that if $j_{d}$ satisfies $\bold{H}(d,m)$, then $\mathcal{T}^{n}j_{d}$ satisfies $\bold{H}(d+2n,m-n)$ for all $n\leq m$. 
\end{remark}

\medskip

In the following remarks, we collect some useful facts about $\ell,K,L,h,\psi$ and $j_{d}$.

\begin{remark}\label{r:assum}
Using \eqref{H:s},  we see  that 
\begin{align}\label{ll1}
\liminf_{r\to\infty}\ell(r)>0,
\end{align}
which is essential to obtain estimations of heat kernel for large time.
\end{remark}

\begin{remark}\label{r:relations}
(i) Using inequalities (6) and (7) in \cite{BGR14a}, we can check that there exists a constant $C_{0}(<1)$ depends only on $d$ such that
\begin{equation}\label{eqn 05.27.15:40}
C_{0}h(r)\le \psi(r^{-1})\le 2h(r) \quad \forall\, 0<r<\infty.
\end{equation}
 Also, it is known that (see \cite[Lemma 2.1]{cho21heat}) there exists $c_{1}>0$ such that
\begin{equation}\label{eqn 08.31.14:07}
L(r) \leq h(r) \leq c_{1}L(r) \quad \forall\, 0<r\leq 1.
\end{equation}
Moreover, for any $r>0$, and $0<c<1$, 
\begin{equation}\label{eqn 05.12:15:07}
h(cr) \leq c(d)\kappa_{1,0}^{-1} c^{-2} r^{-2} \int_{\bR^{d}} \left( c^{2} r^{2} \wedge |y|^{2} \right) j_{d}(|y|) dy  \leq \kappa_{2,0}\kappa_{1,0}^{-1} c^{-2} h(r).
\end{equation}

(ii) Under \eqref{H:s} and \eqref{H:l}, it is known that (see inequalities (2.2) and (2.3) in \cite{cho21heat}) 
\begin{align}\label{Kellcomp}
K(r)\asymp \ell(r^{-1})\asymp r^d j_d(r)\quad\text{for}\;\;r>0,
\end{align}
By \eqref{H:l},
$$
L(r)=\int^{\infty}_{r}s^{-1}\ell(s^{-1})ds\le C_{1}^{-1}r^{\delta_3}\ell(r^{-1})\int^{\infty}_{r}s^{-1-\delta_3}ds=C_{1}^{-1}\delta_{3}^{-1}\ell(r^{-1}) \quad \forall \, r\geq 1.
$$
This and \eqref{Kellcomp} imply 
and
$$
h(r)\asymp K(r) \asymp \ell (r^{-1}) \quad  \forall \, r\ge1.
$$
Combining this with \eqref{eqn 05.27.15:40} we have
\begin{align}\label{eqn 09.21.13:27}
\psi(r^{-1})\asymp h(r) \asymp K(r) \asymp \ell(r^{-1}) \quad \forall\, r\ge1.
\end{align}
Finally, if $\delta_{1}$ in \eqref{H:s} is positive, then 
\begin{align*}
K(r) \asymp  h(r) \asymp \ell(r^{-1}) \quad \text{for}\; \, r>0.
\end{align*}

(iii) From \eqref{eqn 05.12:15:07}, we can obtain
\begin{equation*}
\frac{R}{r} \leq \kappa_{1,0}^{-1}\kappa_{2,0} \left( \frac{h^{-1}(r)}{h^{-1}(R)} \right)^{2} \quad \forall\, 0<r<R<\infty.
\end{equation*}

(iv) By a direct calculation, we can check that
\begin{align}
K(r/2) = (r/2)^{-2} \int_{0}^{r/2}s\ell(s^{-1})ds 
 \leq  (r/2)^{-2} \int_{0}^{r}s\ell(s^{-1})ds   \leq  4K(r).   \label{eqn 08.23.11:50}
\end{align}
\end{remark}

\begin{remark}\label{rmk 09.15.18:04}
(i) Since $h$ is decreasing, \eqref{eqn 08.12.17:59} is equivalent to 
$$
\sup_{0<r} h(r)e^{-ah(r)/\ell(r^{-1})} \leq C(a).
$$ 

(ii) By \cite[Remark 2.5 (iii)]{KP21}, we see that \eqref{eqn 08.12.17:59} is equivalent to the following: for any $a>0$, there exists $C(a)>0$ such that
\begin{equation}\label{con_ii}
\sup_{r>1} \int^{r}_{1}\frac{\ell(s)}{s}ds\cdot\exp{\left( -\frac{a}{\ell(r)}\int^{r}_{1}\frac{\ell(s)}{s}ds  \right)} \leq C(a).
\end{equation}

(iii) From \cite[Lemma A.1 (ii), (iv)]{KP21}, we see that $\log{(1+r^{b})}$ satisfies \eqref{con_ii} for any $b>0$.
\end{remark}

\begin{example}\label{exm 10.05.14:52}
We give some examples satisfying our assumptions. See \cite[Example 2.1]{KP21} for information on corresponding characteristic exponents.

(i) Let $\alpha\in(0,2)$, $\ell(r) = r^{\alpha}$, and let $j_{d}(r) := r^{-d}\ell(r^{-1})$. Then $\ell$ satisfies Assumption \ref{ass bernstein} with $\delta_{1}=\delta_{2}=\delta_{3}=\delta_{4}=\alpha$, and Assumption \ref{ell_con} (ii)--(1). Also, $j_{d}$ satisfies $\bold{H}(d,m)$ for any $m\in\bN_{0}$.

(ii) Let $\alpha \in (0,2)$, and $\ell(r) = 1\wedge r^{\alpha}$. Define
$$
j_{d}(r) :=r^{-d}\left(\frac{r^{-\alpha / \lceil \alpha \rceil}}{1+r^{-\alpha / \lceil \alpha \rceil}}\right)^{\lceil \alpha \rceil},
$$
where $\lceil \alpha \rceil$ is the smallest integer greater than or equal to $\alpha$. Then $\ell$ satisfies Assumption \ref{ass bernstein} with $\text{{\boldmath$\delta$}}=(0,0,\alpha,\alpha)$, and Assumption \ref{ell_con} (i). Also, using Lemma \ref{lem 08.31.16:04}, and Lemma \ref{lem 08.31.16:04-2} (i), we can check that $j_{d}$ satisfies $\bold{H}(d,m)$ for any $m\in \bN_{0}$.

(iii) Let $\alpha\in (0,2)$, and $\ell(r) = \log{(1+r^{\alpha})}$. Define
$$
j_{d}(r) := r^{-d}\log{(1+r^{-\alpha / \lceil \alpha \rceil})} \left(\frac{r^{-\alpha / \lceil \alpha \rceil}}{1+r^{-\alpha / \lceil \alpha \rceil}}\right)^{\lceil \alpha \rceil-1}.
$$ 
Then $\ell$ satisfies Assumption \ref{ass bernstein} with $\text{{\boldmath$\delta$}}=(0,1,\alpha,\alpha)$, and Assumption \ref{ell_con} (ii)--(2). Also, using Lemma \ref{lem 08.31.16:04}, and Lemma \ref{lem 08.31.16:04-2}, we can check that $j_{d}$ satisfies $\bold{H}(d,m)$ for any $m\in \bN_{0}$.

(iv) For $n\in \bN$, let
$$
f_{n}(r) = \begin{cases} \frac{1}{n!}(1-r)^{n} & \text{for} \quad 0\leq r \leq 1
\\
0 & \text{for} \quad 1\leq r.
\end{cases}
$$
Then we can check that $f_{n}$ is $n$-times differentiable, and $f^{(n+1)}_{n}(1)$ does not exist. Using this, we can check that $\tilde{j}_{d}(r):=j_{d}(r)(1+f_{n}(r))$, where $j_{d}$ is taken from the above examples is not infinitely differentiable, and hence it cannot be the jump kernel of subordinate Brownian motion. 
\end{example}

The following theorem is the main result of this article. 

\begin{theorem} \label{main theorem2}
Let $d \in \bN$, $p,q\in(1,\infty)$, $\gamma \in \bR$, and $T\in(0,\infty)$. Suppose  Assumption \ref{ass bernstein}, and Assumption \ref{ell_con} hold. Also, suppose that $j_{d}$ satisfies $\bold{H}(d,4)$. Then for any  $f\in \bH_{q,p}^{\psi,\gamma}(T)$, the equation
\begin{equation}\label{mainequation2}
\partial_{t}u = \mathcal{L}u + f,\quad t>0\,; \quad u(0,\cdot)=0
\end{equation}
has a unique solution $u$ in the class $\bH_{q,p,0}^{\psi,\gamma+2}(T)$, and for the solution $u$ it holds that
\begin{equation} \label{mainestimate2}
\|u\|_{\bH_{q,p}^{\psi,\gamma+2}(T)}\leq C \|f\|_{\bH_{q,p}^{\psi,\gamma}(T)},
\end{equation}
where $C>0$ is a constant depending only on $d,p,q,\gamma,\ell,T, \text{{\boldmath$\delta$}}$ and  $\text{{\boldmath$\kappa$}}_{4}$. Furthermore, we have
\begin{equation}\label{eqn 05.27.14:12}
\| \mathcal{L} u\|_{\mathbb{H}^{\psi,\gamma}_{q,p}(T)} \leq C \| f\|_{\mathbb{H}^{\psi,\gamma}_{q,p}(T)},
\end{equation}
where $C>0$ is a constant  depending only on $d,p,q,\gamma,\ell, \text{{\boldmath$\delta$}}$ and  $\text{{\boldmath$\kappa$}}_{4}$.
\end{theorem}

Theorem \ref{main theorem2} and theory of additive processes deduce the following $L_{p}$-regularity of solution for equations with non-local operators having time-measurable coefficients.

\begin{theorem} \label{main theorem2-1}
Let $d \in \bN$, $p,q\in(1,\infty)$, $\gamma \in \bR$, and $T\in(0,\infty)$. Suppose  Assumption \ref{ass bernstein}, and Assumption \ref{ell_con} hold. Also, suppose that $j_{d}$ satisfies $\bold{H}(d,4)$. Let $a=a(t,y)$ be a measurable function satisfying
\begin{align}\label{eqn 09.05.16:00}
\int_{r<|y|<R}  y \, a(t,y) \, j_{d}(y) \, dy = 0 \quad \forall\, 0<r<R<\infty,
\end{align}
$$
0<a_{0}\leq a(t,y) \leq a_{1}<\infty \quad \forall (t,y) \in (0,\infty)\times \bR^{d},
$$
and define the operator $\mathcal{L}^{a}$ as
\begin{align}\label{eqn 03.15.14:09}
\mathcal{L}^{a}u(t,x) = \int_{\bR^{d}} \left( u(x+y) - u(x) -\nabla u(x) \cdot y \mathbf{1}_{|y|\leq1} \right) a(t,y)j_{d}(y)dy.
\end{align}
Then for any  $f\in \bH_{q,p}^{\psi,\gamma}(T)$, the equation
\begin{equation*}
\partial_{t}u = \mathcal{L}^{a}u + f,\quad t>0\,; \quad u(0,\cdot)=0
\end{equation*}
has a unique solution $u$ in the class $\bH_{q,p,0}^{\psi,\gamma+2}(T)$, and for the solution $u$ it holds that
\begin{equation*}
\|u\|_{\bH_{q,p}^{\psi,\gamma+2}(T)}\leq C \|f\|_{\bH_{q,p}^{\psi,\gamma}(T)},
\end{equation*}
where $C>0$ is a constant depending only on $d,p,q,\gamma,\ell,T, a_{0},a_{1}, \text{{\boldmath$\delta$}}$ and  $\text{{\boldmath$\kappa$}}_{4}$. Furthermore, we have
\begin{equation*}
\| \mathcal{L}^{a} u\|_{\mathbb{H}^{\psi,\gamma}_{q,p}(T)}+ \| \mathcal{L} u\|_{\mathbb{H}^{\psi,\gamma}_{q,p}(T)} \leq C \| f\|_{\mathbb{H}^{\psi,\gamma}_{q,p}(T)},
\end{equation*}
where $C>0$ is a constant  depending only on $d,p,q,\gamma,\ell, \text{{\boldmath$\delta$}},a_{0},a_{1}$ and  $\text{{\boldmath$\kappa$}}_{4}$.
\end{theorem}
\begin{proof}
We can prove the theorem by directly following the proof of \cite[Theorem 2.8]{CKP23anisotropic} with Theorem \ref{main theorem2}.
\end{proof}

\begin{remark}
Here we remark that the proof of \cite[Theorem 2.8]{CKP23anisotropic} requires the symmetry of the coefficient \eqref{eqn 09.05.16:00} since the approach is based on the $L_{p}$-boundedness of the Fourier multiplier
\begin{align*}
m(t,\xi):=\frac{\int_{\bR^{d}} \left( 1-\cos{(y\cdot \xi)} \right) a(t,y) j_{d}(y)dy}{\int_{\bR^{d}} \left( 1-\cos{(y\cdot \xi)} \right) j_{d}(y)dy}
\end{align*}
which becomes the Fourier multiplier of the operator $\mathcal{L}^{a} \mathcal{L}^{-1}$ under \eqref{eqn 09.05.16:00}. 
\end{remark}

\begin{remark}
Suppose that $\mathcal{L}=\mathcal{L}_{Y}$ is an operator with the Fourier multiplier $\xi\mapsto\phi(|\xi|^{2})$, where $\phi$ is a Bernstein function. (Recall Example \ref{exm 04.04.17:27}.) Then, due to \cite{kim2019L,KW04}, the inequalities \eqref{mainestimate2} and \eqref{eqn 05.27.14:12} in Theorem \ref{main theorem2} associated with $\mathcal{L}_{Y}$ hold without any assumptions. By using these results, when $a(t,y)$ is non-trivial, then Theorem \ref{main theorem2-1} corresponding to $\mathcal{L}_{Y}^a$ can be obtained similarly. We emphasize that in Theorem \ref{main theorem2}, our non-local operator $\mathcal{L}$ can have the Fourier multiplier $\psi$, which is neither a smooth function nor a function with scaling property.
\end{remark}

\medskip

\mysection{Estimates of the heat kernels and their derivatives}\label{sec3}

In this section, we obtain  sharp bounds of  the heat kernel and its derivative for the equation 
\begin{equation*}
\partial_t u = \mathcal{L}u
\end{equation*} 
under the Assumption \ref{ass bernstein}, and Assumption \ref{ell_con}.

Let $p(t,x)=p_d(t,x)$ be the transition density of $X_t$. Then it is well-known that for any $t>0, x\in \bR^d$,
\begin{equation}\label{eqn 7.20.1}
p_d(t,x)=\frac{1}{(2\pi)^d} \int_{\bR^d} e^{i \xi \cdot x} e^{-t\psi(|\xi|)} \,d\xi.
\end{equation}
Since $X_{t}$ is isotropic, $p_{d}(t,x)$ is rotationally invariant in $x$ (i.e. $p_{d}(t,x)=p_{d}(t,|x|)$). We put $p_{d}(t,r):=p_{d}(t,x)$  if $r=|x|$ for notational convenience. Since $X_{t}$ is unimodal, $r\mapsto p(t,r)$ is a decreasing function. Moreover, $p(t,x)\le p(t,0)\in(0,\infty]$ for $t>0$ and $x\in\R^d$.

The following lemma gives off a diagonal type upper bound for the heat kernel. The result holds for all isotropic unimodal L\'evy processes. 

\begin{proposition}[{\cite[Theorem 5.4]{GRT19}}]\label{p:uhk_jump}
For any $(t,x)\in (0,\infty)\times(\bR^{d}\setminus\{0\})$, we have
\begin{equation*}
p_{d}(t,x) \leq C t|x|^{-d}K(|x|),
\end{equation*}
where the constant $C>0$ depends only on $\kappa_{2,0},d$.
\end{proposition}

Let $\bar{\ell}(r) := \sup_{s\leq r} \ell(s)$. Then $\bar{\ell}$ is an increasing continuous function. Using Assumption \ref{ass bernstein}, we can construct a strictly increasing continuous function $\ell^{\ast}$ such that $\bar{\ell}(r) \asymp \ell^{\ast}(r)$ for all $r>0$ (see \cite[Lemma A.1]{KP21}).  Therefore, if $\ell$ satisfies Assumption \ref{ell_con} (ii), then there exists $C_{3}$ such that $\ell^{\ast}(r) \leq C_{3} \ell(r)$ for all $r>0$. In this case, we also denote $\ell^{-1}$ the inverse function of $\ell^{\ast}$.

The following lemma is an extension of \cite[Lemma 2.7]{cho21heat} in the sense that we can choose any $a\geq a_{d}$ in the lemma. Note that such freedom of choice is crucial for estimations of derivatives of heat kernel.

\begin{lemma}\label{lem 08.04.17:37}
Suppose $\ell$ satisfies  Assumption \ref{ell_con} (ii) and let $a_{d}:=2dC_{3}/C_{0}$, where $C_{0} \in(0,1)$ comes from \eqref{eqn 05.27.15:40}.  Then for any $a\geq a_{d}$, and $t\leq a/\ell^{\ast}(3)$, we have
$$
p_{d}(t,x) \leq p_{d}(t,0) \leq C [\ell^{-1}(a/t)]^{d} \exp{\left(-\frac{C_{0}}{4c_{1}a}th((\ell^{-1}(a/t))^{-1})\right)}
$$
where the constant $C$ depends only on $d,a$, and  $c_{1}>0$ comes from \eqref{eqn 08.31.14:07}.
\end{lemma}

\begin{proof}
We can obtain the desired result by following the proof of \cite[Lemma 2.7]{cho21heat}.
\end{proof}

For any $a>0$ and $r,t>0$, we define
\begin{align}\label{d:theta}
\theta(a,r,t)=\theta_a(r,t):=r\vee(\ell^{-1}(a/t))^{-1}.
\end{align}

\begin{proposition}\label{p:hku}
Suppose $\ell$ satisfies the condition in Assumption \ref{ell_con} (ii) and $\theta$ is given by \eqref{d:theta}. Also, take constant $a_{d}$ from Lemma \ref{lem 08.04.17:37}. Then
for any $a\geq a_{d}$ and $\varepsilon>0$, the following holds: for any $T>0$, there exist $c_{2},b_{0}>0$ depending only on $\kappa_{1,0},\kappa_{2,0},d,a$ and $T$ such that
\begin{align*}
c_{2}^{-1}t\frac{K(\theta_{\varepsilon}(|x|,t))}{[\theta_{\varepsilon}(|x|,t)]^d}\exp{\left(-b_{0}th(\theta_{\varepsilon}(|x|,t)) \right)}  
 \leq  p_{d}(t,x)  
 \leq c_{2}t\frac{K(\theta_{a}(|x|,t))}{[\theta_{a}(|x|,t)]^d}\exp{\left(-\frac{C_{0}}{4c_{1}a}th(\theta_{a}(|x|,t)) \right)} 
\end{align*}
holds for all $(t,x)\in (0,T]\times \bR^{d}$. 
\end{proposition}
\begin{proof}
Using \cite[Proposition 5.3]{GRT19} and \eqref{Kellcomp} we have
\begin{align*}
c_{2} t\frac{K(|x|)}{|x|^{d}} \exp{\left( -b_{0} t h(|x|) \right)}\leq c t j_{d}(|x|)\exp{\left( -b_{0} t h(|x|) \right)} \leq  p_{d}(t,x).
\end{align*}
Therefore, since $\theta_{\varepsilon}(|x|,t) \geq |x|$ for any $\varepsilon,x,t$, we have
\begin{align*}\label{eqn 09.26.17:24}
 c_{2}^{-1} t\frac{K(\theta_{\varepsilon}(|x|,t))}{[\theta_{\varepsilon}(|x|,t)]^d}\exp{\left(-b_{0}th(\theta_{\varepsilon}(|x|,t)) \right)} 
\leq p_{d}(t,\theta_{\varepsilon}(|x|,t)) \leq p_{d}(t,x),
\end{align*}
and this certainly shows the lower estimation. For the upper estimation, follow the argument in  \cite[Proposition 2.9, Corollary 2.10]{cho21heat} with Lemma \ref{lem 08.04.17:37}.

\end{proof}

The following lemma gives an upper bound of $p_{d}(t,x)$ for sufficiently large $t>0$. For the proof, see \cite[Lemma 3.2]{KP21}.

\begin{lemma}\label{l:hke_largetime}
There exist $t_{1}=t_{1}(d,\kappa_{1,0},\kappa_{2,0},\ell,\text{{\boldmath$\delta$}})>0$ and $C>0$ depending only on $t_{1}$ such that for all $t\geq t_{1}$ and $x\in\bR^{d}$,
\begin{align*}
p_d(t,x)\le C\left( (h^{-1}(t^{-1}))^{-d}\wedge t\frac{K(|x|)}{|x|^d}\right).
\end{align*}
\end{lemma}

\begin{remark}\label{rmk 07.27.10:03}
If $j_{d}$ satisfies $\bold{H}(d,1)$, then using \cite[Theorem 1.5]{kul2016gradient}, we can obtain $(d+2)$-dimensional isotropic unimodal L\'evy process $\tilde{X}_{t}$ with the same characteristic exponent $\psi(|\xi|)=\psi_X(|\xi|)$, whose transition density $p_{d+2}(t,x)=p_{d+2}(t,|x|)$ is radial,  radially decreasing in $x$ and satisfies
\begin{equation*}
p_{d+2}(t,r) = -\frac{1}{2\pi r} \frac{d}{dr}p_{d}(t,r) \quad\text{for}\;\; r>0.
\end{equation*}
This implies that for $t>0$ and $x\in\R^d$,
\begin{equation*}
|D_x p_{d}(t,x)|  \leq 2\pi |x| p_{d+2}(t, |x|).
\end{equation*}
Also, by inspecting the proof of \cite[Theorem 1.5]{kul2016gradient}, we can also find that the L\'evy density $\tilde{j}_{d+2}$ of $\tilde{X}_{t}$ is given by
\begin{align*}
\tilde{j}_{d+2}(r) = \frac{1}{2\pi }\mathcal{T}j_{d}(r):= -\frac{1}{2\pi } \frac{1}{r} \frac{d}{dr} j_{d}(r) \quad\text{for}\;\; r>0.
\end{align*}
If $j_{d}$ satisfies $\bold{H}(d,2)$, then using Remark \ref{rmk 09.27.14:11} (ii) and Assumption \ref{ass bernstein} (i) and  \cite[Theorem 1.5]{kul2016gradient} again, we can deduce that $p_{d+2}(t,x)$ is also differentiable in $x$, and we have
$$
|D^{2}_{x}p_{d}(t,x)| \leq C \left( p_{d+2}(t,|x|) + |x|^{2}p_{d+4}(t,|x|) \right),
$$
where $p_{d+4}$ is a heat kernel for $(d+4)$-dimensional isotropic L\'evy process with characteristic exponent $\psi$ and L\'evy density $\tilde{j}_{d+4}=(4\pi)^{-2}\mathcal{T}^{2}j_{d}$. Continuing, for any $m\in \bN$, if $j_{d}$ satisfies $\bold{H}(d,m)$, we have
\begin{equation}\label{eqn 08.04.18:41}
|D^{m}_{x}p_{d}(t,x)| \leq C \sum_{m-2k \geq 0,k\in \bN_{0}} |x|^{m-2k}p_{d+2(m-k)}(t,|x|).
\end{equation}
\end{remark}

Using the above remark, we have the following theorem for estimates of $D^{m}_{x}p_{d}(t,x)$.

\begin{theorem}\label{thm 07.25.16:53}
Let $m\in \bN_{0}$, and suppose that $j_{d}$ satisfies $\bold{H}(d,m)$. Then we have the following;

(i) Suppose $\ell$ satisfies the condition in Assumption \ref{ell_con} (ii) and $\theta$ is given by \eqref{d:theta}.  
Then, there exists constant $a_{d,m}>0$ such that the following holds: for any $a\geq a_{d,m}$, and $T>0$, there exists $C$ depending only on $\text{{\boldmath$\kappa$}}_{m},d,m$ and $T$ such that
\begin{equation*}
|D^{m}_{x}p_{d}(t,x)| \leq C  t\frac{K(\theta_{a}(|x|,t))}{[\theta_{a}(|x|,t)]^{d+m}}\exp{(-C^{-1}th(\theta_{a}(|x|,t)))}
\end{equation*}
holds for all $(t,x)\in (0,T]\times \bR^{d}$

(ii) There exist $t_{1}=t_{1}(d,\text{{\boldmath$\kappa$}}_{m},\ell,\text{{\boldmath$\delta$}},m)>0$ and $C>0$ depending only on $t_{1}$ such that for all $t\geq t_{1}$ and $x\in\bR^{d}$,
\begin{align*}
|D^{m}_{x}p_d(t,x)| \le C \sum_{m-2k \geq 0,k\in \bN_{0}}  |x|^{m-2k} \left( (h^{-1}(t^{-1}))^{-d-2(m-k)}\wedge t\frac{K(|x|)}{|x|^{d+2(m-k)}}\right).
\end{align*}
\end{theorem}

\begin{proof}

Using \eqref{eqn 08.04.18:41} with Lemma \ref{l:hke_largetime}, the second assertion easily follows. 

For the first assertion, let $a_{d+2(m-k)}=2(d+2(m-k))C_{3}/C_{0}$ ($0 \leq m-2k \leq m$) be taken from Lemma \ref{lem 08.04.17:37}. Since $a_{d+2(m-k)} \leq a_{d+2m}$, using \eqref{eqn 08.04.18:41} and Proposition \ref{p:hku}, for any $a\geq a_{d+2m}$, we have
\begin{align*}
|D^{m}_{x}p_{d}(t,x)| 
&\leq C \sum_{m-2k \geq 0,k\in \bN_{0}} |x|^{m-2k}p_{d+2(m-k)}(t,|x|)
\\
& \leq C \sum_{m-2k \geq 0,k\in \bN_{0}} |x|^{m-2k}
\times t\frac{K(\theta_{a}(|x|,t))}{[\theta_{a}(|x|,t)]^{d+2(m-k)}}\exp{\left(-\frac{C_{0}}{4c_{1}a}th(\theta_{a}(|x|,t)) \right)}
\\
& \leq C \sum_{m-2k \geq 0,k\in \bN_{0}} [\theta_{a}(|x|,t)]^{m-2k}
 \times t\frac{K(\theta_{a}(|x|,t))}{[\theta_{a}(|x|,t)]^{d+2(m-k)}}\exp{\left(-\frac{C_{0}}{4c_{1}a}th(\theta_{a}(|x|,t)) \right)}
\\
& \leq C t\frac{K(\theta_{a}(|x|,t))}{[\theta_{a}(|x|,t)]^{d+m}}\exp{\left(-\frac{C_{0}}{4c_{1}a}th(\theta_{a}(|x|,t)) \right)}.
\end{align*}
Therefore, by taking $a_{d,m}=a_{d+2m}$ we prove the first assertion. The theorem is proved.
\end{proof}

The following lemma is a collection of computations used in the rest of the article.

\begin{lemma}\label{rmk 05.07.16:12}
Let $a, b,T>0$ and $d\in\bN$.\\
(i) There exists $C=C(b,d)>0$ such that for all $t>0$
\begin{equation*}
\begin{aligned}
\int_{\bR^{d}} t\frac{K(|x|)}{|x|^{d}}\exp{(-bth(|x|))}\,dx \leq C.
\end{aligned}
\end{equation*}
(ii) Suppose $\ell$ satisfies the condition in Assumption \ref{ell_con} (ii) and $\theta$ is given by \eqref{d:theta}. There exists $C=C(a,b,d)>0$ such that for all $t>0$
\begin{equation*}
\begin{aligned}
\int_{\bR^{d}} t\frac{K(\theta_{a}(|x|,t))}{[\theta_{a}(|x|,t)]^d}\exp{(-b th(\theta_{a}(|x|,t)))}\,dx \leq C.
\end{aligned}
\end{equation*}
(iii) Suppose $\ell$ satisfies the condition in Assumption \ref{ell_con} (ii) and $\theta$ is given by \eqref{d:theta}. Then for any $k,m\in \bN_{0}$, $s\in(0,1)$ and $x,y\in \bR^{d}$ such that $|y|\leq \theta_{a}(|x|,t)/2$, we have
\begin{align}\label{eqn 09.02.14:03}
 t^{1-k}\frac{K(\theta_{a}(|x+sy|,t))}{\theta_{a}(|x+sy|,t)^{d+m}}\exp{(-b th(\theta_{a}(|x+sy|,t)))} 
 \leq C t^{-k}\frac{1}{\theta_{a}(|x|,t)^{d+m}}\exp{(-C^{-1} th(\theta_{a}(|x|,t)))},
\end{align}
where the constant $C$ depends only on $b,d,m$.

(iv) Let $m\in\bN_{0}$, $a\geq a_{d+m}$, and let $\theta$ be given by \eqref{d:theta}. Suppose $\ell$ satisfies the condition in Assumption \ref{ell_con} (ii). Then for any $(t,x),(t,y)\in (0,T]\times \bR^{d}$ with $|x|\leq |y|$, we have
\begin{align}\label{eqn 09.05.16:39}
t\frac{K(\theta_{a}(|y|,t))}{\theta_{a}(|y|,t)^{d+m}}\exp{\left( -bth(\theta_{a}(|y|,t)) \right)} 
 \leq C t\frac{K(\theta_{\tilde{a}}(|x|,t))}{\theta_{\tilde{a}}(t,|x|)^{d+m}}\exp{\left( -C^{-1}th(\theta_{\tilde{a}}(|x|,t)) \right)}
\end{align}
where the constant $C>0$ depends only on $\text{{\boldmath{$\kappa$}}}_{m},a,b,d,m$ and $T$, $\tilde{a}=(\frac{b_{0}}{b} \vee 1)a$ and $b_{0}$ (corresponding to $d+m$) comes from Proposition \ref{p:hku}.
\end{lemma}
\begin{proof}
(i), (ii) See \cite[Lemma 3.1]{KP21}

(iii) Let $s\in(0,1)$ and $|y| \leq \theta_{a}(|x|,t)/2$. One can check that if $\theta_{a}(|x|,t)= (\ell^{-1}(a/t))^{-1}$, then
\begin{align*}
\theta_{a}(|x|,t)  \leq (|x+sy| \vee (\ell^{-1}(a/t))^{-1})  = \theta_{a}(|x+sy|,t)
 \leq ((|x|+\theta_{a}(|x|,t))\vee (\ell^{-1}(a/t))^{-1})  \leq 2\theta_{a}(|x|,t)  
\end{align*}
and if $\theta_{a}(|x|,t)=|x|$ then $\theta_{a}(|x|,t)/2 \leq \theta(|x+sy|) \leq 2\theta_{a}(|x|,t)$. Hence, we have
\begin{align*}
&t^{1-k}\frac{K(\theta_{a}(|x+sy|,t))}{\theta_{a}(|x+sy|,t)^{d+m}}\exp{(-b^{-1}th(\theta_{a}(|x+sy|,t))))}
\\
&\leq C t^{1-k}\frac{h(\theta_{a}(|x|,t)/2)}{[\theta_{a}(|x|,t)/2]^{d+m}}\exp{(-C^{-1}th(2\theta_{a}(|x|,t)))}\mathbf{1}_{\theta_{a}(|x|,t)=|x|}
\\
&\quad + C t^{1-k}\frac{h(\theta_{a}(|x|,t))}{\theta_{a}(|x|,t)^{d+m}}\exp{(-C^{-1}th(2\theta_{a}(|x|,t)))}\mathbf{1}_{\theta_{a}(|x|,t)=(\ell^{-1}(a/t))^{-1}}
\\
&\leq C  t^{-k} \frac{1}{\theta_{a}(|x|,t)^{d+m}} \exp{(-C^{-1}th(\theta_{a}(|x|,t)))}.
\end{align*}
Thus we have \eqref{eqn 09.02.14:03}.

(iv) Suppose that $x,y\in\bR^{d}$ and let $|x|\leq |y|$. Then if we let $ba/b_{0} = \bar{a}$ and $bt/b_{0}=\bar{t}$, by Proposition \ref{p:hku} (with $d+m$ in place of $d$), and the relation
\begin{equation}\label{eqn 09.13.13:46}
\theta_{a}(r,t) = (\ell^{-1}(a/t))^{-1} \vee r = (\ell^{-1}(bab_{0}/btb_{0}))^{-1} \vee r = \theta_{\bar{a}}(r,\bar{t}),
\end{equation}
we can check that
\begin{align}\label{eqn 09.13.16:14}
 t\frac{K(\theta_{a}(|y|,t))}{\theta_{a}(|y|,t)^{d+m}}\exp{\left(-b th(\theta_{a}(|y|,t)) \right)} 
&= \left( \frac{b_{0}}{b} \right) \left( \frac{b}{b_{0}} \right)  t\frac{K(\theta_{a}(|y|,t))}{\theta_{a}(|y|,t)^{d+m}}\exp{\left(- b_{0} \frac{b}{b_{0}}  th(\theta_{a}(|y|,t)) \right)}  \nonumber
\\
&= \left( \frac{b_{0}}{b} \right) \bar{t} \frac{K(\theta_{\bar{a}}(|y|,\bar{t}))}{\theta_{\bar{a}}(|y|,\bar{t})^{d+m}}\exp{\left(- b_{0} \bar{t} h(\theta_{\bar{a}}(|y|,\bar{t})) \right)}  \nonumber
\\
&\leq C \left( \frac{b_{0}}{b} \right) p_{d+m}(\bar{t},y) \leq C \left( \frac{b_{0}}{b} \right) p_{d+m}(\bar{t},x),
\end{align}
where $p_{d+m}$ is a heat kernel for isotropic unimodal $d+m$-dimensional L\'evy process with L\'evy measure $r^{-m}j_{d}(r)$. If $b_{0}\leq b$, then $\bar{a} = ba/b_{0} \geq a=\tilde{a}$. Therefore, applying Proposition \ref{p:hku} and \eqref{eqn 09.13.13:46} again, we have
\begin{align*}
 t\frac{K(\theta_{a}(|y|,t))}{\theta_{a}(|y|,t)^{d+m}}\exp{\left(-b th(\theta_{a}(|y|,t)) \right)}
&\leq C \left( \frac{b_{0}}{b} \right) p_{d+m}(\bar{t},x)
\\
&\leq C  \left( \frac{b_{0}}{b} \right) \bar{t} \frac{K(\theta_{\bar{a}}(|x|,\bar{t}))}{\theta_{\bar{a}}(|x|,\bar{t})^{d+m}}\exp{\left(-\frac{C_{0}}{4c_{1}\bar{a}}\bar{t}h(\theta_{\bar{a}}(|x|,\bar{t})) \right)}
\\
&\leq C t  \frac{K(\theta_{a}(|x|,t))}{\theta_{a}(|x|,t)^{d+m}}\exp{\left(-\frac{C_{0}}{4c_{1}a}th(\theta_{a}(|x|,t)) \right)}.
\end{align*}
If $b_{0}>b$ (i.e. $\tilde{a} = \frac{b_{0}}{b}a \geq a$), then using the above argument with
$\theta_{a}(|x|,\bar{t}) = (\ell^{-1}(b_{0}a/bt))^{-1}\vee |x| = \theta_{\tilde{a}}(|x|,t)$,
we have
\begin{align*}
 t\frac{K(\theta_{a}(|y|,t))}{[\theta_{a}(|y|,t)]^{d+m}}\exp{\left(-b th(\theta_{a}(|y|,t)) \right)}
&\leq C \left( \frac{b_{0}}{b} \right) \bar{t}\frac{K(\theta_{a}(|x|,\bar{t}))}{\theta_{a}(|x|,\bar{t})^{d+m}}\exp{\left(-\frac{C_{0}}{4c_{1}a}\bar{t}h(\theta_{a}(|x|,\bar{t})) \right)}
\\
&= C t \frac{K(\theta_{\tilde{a}}(|x|,t))}{\theta_{\tilde{a}}(|x|,t)^{d+m}}\exp{\left(-\frac{bC_{0}}{4b_{0}c_{1}a}th(\theta_{\tilde{a}}(|x|,t)) \right)}
\\
& \leq C t \frac{K(\theta_{\tilde{a}}(|x|,t))}{\theta_{\tilde{a}}(|x|,t)^{d+m}}\exp{\left(-C^{-1}th(\theta_{\tilde{a}}(|x|,t)) \right)}.
\end{align*}
The lemma is proved.
\end{proof}

\begin{remark}\label{rmk 09.13.15:47}
(i) We can also show that for any $m,k\in \bN_{0}$, $b>0$ and $x,y\in \bR^{d}\setminus\{0\}$ with $|y|\leq |x|/2$
\begin{align*}
 t^{1-k}\frac{K(|x+y|)}{|x+y|^{d+m}}\exp{(-b th(|x+y|))}
& \leq t^{1-k}\frac{h(|x+y|)}{|x+y|^{d+m}}\exp{(-b th(|x+y|))}
\\
& \leq C t^{-k}\frac{1}{|x|^{d+m}}\exp{(-C^{-1} th(2|x|))}
\\
&\leq C t^{-k}\frac{1}{|x|^{d+m}}\exp{(-C^{-1} th(|x|))}
\end{align*}
where the constant $C>0$ depends only on $d,b,m$.

(ii) Due to \cite[Corollary 2.13]{cho21heat}, if $\ell$ satisfies Assumption \ref{ass bernstein} (i), then we have the following correspondence of \eqref{eqn 09.13.16:14};
\begin{align*}
 t\frac{K(|y|)}{\theta_{a}(|y|,t)^{d+m}}\exp{\left(-b th(|y|) \right)} 
\leq C \left( \frac{b_{0}}{b} \right) p_{d+m}(ct,y)  
\leq C \left( \frac{b_{0}}{b} \right) p_{d+m}(ct,x),
\end{align*}
where the constants $C,c>0$ depend only on $b,d,m,T,\kappa_{1,m}$ and $\kappa_{2,m}$. Hence, if we apply \cite[Corollary 2.13]{cho21heat} again for any $x,y\in \bR^{d}\setminus\{0\}$ such that $|x|\leq |y|$ we have
\begin{align*}
t\frac{K(|y|)}{|y|^{d+m}}\exp{\left( -bth(|y|) \right)}   \leq C t\frac{K(|x|)}{|x|^{d+m}}\exp{\left( -C^{-1}th(|x|) \right)},
\end{align*}
where the constant $C$ depends only on $b,d,m,T,\kappa_{1,m}$ and $\kappa_{2,m}$.
\end{remark}

\begin{lemma} \label{pfracderivativeestimate}
Let $m\in \bN$, and let $\theta$ be given by \eqref{d:theta}. Suppose $j_{d}$ satisfies $\bold{H}(d,m+2)$ and $\ell$ satisfies the condition in Assumption \ref{ell_con} (ii).
Then, for any $m\in \bN_{0}$, there exists $\alpha_{m}>0$ such that the following holds: for any $a\geq\alpha_{m}$ and   $T>0$, there exists  $C>0$ depending only on $\text{{\boldmath{$\kappa$}}}_{m+2},d,m,a$ and $T$ such that
\begin{align*}
\left| \mathcal{L}D^{m}_{x}p_{d}(t,x) \right| & \leq  C  \frac{K(\theta_{a}(|x|,t))}{[\theta_{a}(|x|,t)]^{d+m}}\exp{(-C^{-1}th(\theta_{a}(|x|,t)))}
\end{align*}
for all $(t,x)\in (0,T]\times \bR^{d}$.
\end{lemma}

\begin{proof}
Let $T>0$ be fixed and let $t\leq T$. Take $a_{d,m+2}$ from Theorem \ref{thm 07.25.16:53}. Throughout the proof, we abuse the notation  $\theta_{a}=\theta_{a}(r):= \theta_{a}(r,t)$ for any $a\geq a_{d,m+2}$  and $t\leq T$. 

From the definition of $\mathcal{L}$ (recall \eqref{eqn 09.28.11:26}), for any $a\geq a_{d,m+2}$ we have 
\begin{align*}
\left|\mathcal{L} D^{m}_{x}p_{d}(t,\cdot)(x)\right| 
&= \left|\int_{\bR^d} \left(D^{m}_{x}p_{d}(t,x+y)-D^{m}_{x}p(t,x)-\nabla D^{m}_{x}p_{d}(t,x)\cdot y \mathbf{1}_{|y|\leq \theta_{a}(|x|)/2}\right)j_{d} (|y|)dy \right|  
\\
&\leq |D^{m}_{x}p_{d}(t,x)|\int_{|y|>\theta_{a}(|x|)/2} j_{d} (|y|)dy
+ \left| \int_{|y|>\theta_{a}(|x|)/2} D^{m}_{x} p_{d}(t,x+y)j_{d} (|y|)dy\right|  
\\
& \quad+ \int_{\theta_{a}(|x|)/2>|y|} \int_0^1 \left| D^{m+1}_{x}p_{d}(t,x+sy)- D^{m+1}_{x}p_{d}(t,x)\right| |y|j_{d}(|y|) dsdy  
\\
&=: |D^{m}_{x}p_{d}(t,x)|\times I + II + III.  
\end{align*}
For $I$ by \eqref{eqn 05.12:15:07} we have
\begin{align*}
I  \leq C \int_{\theta(|x|)_{a}/2}^{\infty} r^{-1} \ell(r^{-1}) dr  
= C L(\theta_{a}(|x|)/2) \leq C h(\theta_{a}(|x|)/2) \leq C h(\theta_{a}(|x|)).
\end{align*}
This together with Theorem \ref{thm 07.25.16:53} yields that for any $a\geq a_{d,m+2}$
\begin{align}\label{eqn 05.04.12:07}
|D^{m}_{x}p_{d}(t,x)|\times I &\leq  Ct  h(\theta_{a}(|x|)) \frac{K(\theta_{a}(|x|))}{\theta_{a}(|x|)^{d+m}}\exp{\left(-\frac{C_{0}}{4c_{1}a}th(\theta_{a}(|x|)) \right)} 
 \leq C \frac{K(\theta_{a}(|x|))}{\theta_{a}(|x|)^{d+m}}\exp{\left(-C^{-1}th(\theta_{a}(|x|)) \right)}.
\end{align}
For $III$, by the fundamental theorem of calculus,
\begin{align*}
III \leq& C(d) \int_{\theta_{a}(|x|)/2>|y|} \int_0^1 \int_0^1 \left|D^{m+2}_{x}p_{d}(t,x+usy)\right| |y|^2j_{d}(|y|) dudsdy.
\end{align*}
By Theorem \ref{thm 07.25.16:53} (i)  and \eqref{eqn 09.02.14:03}, we have
\begin{align*}
|D^{m+2}_{x}p_{d}(t,x+usy)| \leq C t \frac{K(\theta_{a}(|x+usy|)}{\theta_{a}(|x+usy|)^{d+m+2}}\exp{\left( -Cth(\theta_{a}(|x+usy|))  \right)}
\leq  C \frac{1}{\theta_{a}(|x|)^{d+m+2}}\exp{\left( -Cth(\theta_{a}(|x|))  \right)}.
\end{align*}
Therefore, for any $a\geq a_{d,m+2}$, we have
\begin{equation*}\label{eqn 05.04.12:07-1}
III \leq C  \frac{\exp{(-C^{-1}th(\theta_{a}(|x|)))}}{\theta_{a}(|x|)^{d+m+2}} \int_{\theta_{a}(|x|)/2>|y|} |y|^2j_{d}(|y|) dy.
\end{equation*}
Also, we can easily check
\begin{align*}
\int_{\theta_{a}(|x|)/2>|y|} |y|^2 j_{d}(|y|)dy &\leq  C \int_0^{\theta_{a}(|x|)} r \ell(r^{-1}) dr  
= C\theta_{a}(|x|)^{2} K(\theta_{a}(|x|)).
\end{align*}
Hence, for any $a\geq a_{d+2m}$, we have
\begin{align}\label{eqn 07.27.13:48}
III &\leq C \frac{K(\theta_{a}(|x|))}{\theta_{a}(|x|)^{d+m}}\exp{(-C^{-1}th(\theta_{a}(|x|)))}.  
\end{align}

Now we estimate $II$. It is easy to see that 
\begin{equation*}
\begin{aligned}
II \leq  \left| \int_{2\theta_{a}(|x|) \leq |y|}    D^{m}_{x} p_{d}(t,x+y)  j_{d}(|y|)   dy  \right|
 + \left| \int_{\frac{\theta_{a}(|x|)}{2}<|y|<2\theta_{a}(|x|)}  D^{m}_{x}p_{d}(t,x+y)  j_{d}(|y|)  dy \right|
=: II'+II''.
\end{aligned}
\end{equation*}
Note that for any $a>0$, 
\begin{equation*}\label{eqn 08.09.17:57}
 |x+y|\geq |y|-|x|\geq 2\theta_{a}(|x|)-|x| \geq |x| \quad \forall\, |y|\geq 2\theta_{a}(|x|).
\end{equation*}
Thus by Theorem \ref{thm 07.25.16:53} (i) and \eqref{eqn 09.05.16:39}, for any $a\geq a_{d,m+2}$, we have
\begin{align*}
|D^{m}_{x}p_{d}(t,x+y)|   
\leq 
 \tilde{C} t \frac{K(\theta_{a}(|x+y|))}{\theta_{a}(|x+y|)^{d+m}}\exp{(-\tilde{C}^{-1}th(\theta_{a}(|x+y|)))}
\leq 
  C t \frac{K(\theta_{\tilde{a}}(|x|))}{\theta_{\tilde{a}}(|x|)^{d+m}}\exp{(-C^{-1}th(\theta_{\tilde{a}}(|x|)))}
\end{align*}
for any $|y|\geq 2\theta_{a}(|x|)$, where 
\begin{equation}\label{eqn 09.07.15:24}
\tilde{a} = \tilde{C}b_{0}a\geq \tilde{C}b_{0}a_{d,m+2}.
\end{equation}
Here note that $\tilde{C}>0$ depends only on $\text{{\boldmath{$\kappa$}}}_{m+2},d,C_{0},C_{3},m,a$ and $T$ and $\tilde{C}$ can be taken large so that $b_{0}\tilde{C}>1$. Hence, for any $a\geq \tilde{C}b_{0}a_{d,m+2}$ we have
\begin{align*}
II' &\leq t \frac{K(\theta_{\tilde{a}}(|x|))}{\theta_{\tilde{a}}(|x|)^{d+m}} \exp{(-C^{-1}th(\theta_{\tilde{a}}(|x|)))}\int_{|y| \geq 2\theta_{a}(|x|)}  j_{d}(|y|)  dy 
\\
&\leq Ct \frac{K(\theta_{\tilde{a}}(|x|))}{\theta_{\tilde{a}}(|x|)^{d+m}}\exp{(-C^{-1}th(\theta_{\tilde{a}}(|x|)))}  \int_{|y| \geq \theta_{\tilde{a}}(|x|)}  |y|^{-d}\ell(|y|^{-1}) dy  \nonumber
\\  
&\leq C t \frac{K(\theta_{\tilde{a}}(|x|))}{\theta_{\tilde{a}}(|x|)^{d+m}}\exp{(-C^{-1}th(\theta_{\tilde{a}}(|x|)))}   L(\theta_{\tilde{a}}(|x|))  \nonumber
\\
&\leq Cth(\theta_{\tilde{a}}(|x|)) \frac{K(\theta_{\tilde{a}}(|x|))}{\theta_{\tilde{a}}(|x|)^{d+m}}\exp{(-C^{-1}th(\theta_{\tilde{a}}(|x|)))} 
\leq C \frac{K(\theta_{\tilde{a}}(|x|))}{\theta_{\tilde{a}}(|x|)^{d+m}}\exp{(-C^{-1}th(\theta_{\tilde{a}}(|x|)))},  
\end{align*}
where for the second inequality, we used relation $\theta_{\tilde{a}}(|x|) \leq 2\theta_{\tilde{a}}(|x|) \leq 2\theta_{a}(|x|)$.

By the integration by parts,
\begin{equation*}
\begin{aligned}
II'' 
&\leq  \sum_{n=0}^{m-1} \int_{|y| = 2\theta_{a}(|x|)} \left|\frac{d^{n}}{dr^{n}}j_{d} \right|(|y|) |D^{m-1-n}_{x}p_{d}(t,x+y)| dS
\\
&\quad + \sum_{n=0}^{m-1} \int_{|y| = \theta_{a}(|x|)/2} \left|\frac{d^{n}}{dr^{n}}j_{d} \right|(|y|) |D^{m-1-n}_{x}p_{d}(t,x+y)| dS
\\
& \quad +  \int_{\theta_{a}(|x|)/2\leq |y| \leq 2\theta_{a}(|x|)} \left|\frac{d^{m}}{dr^{m}}j_{d} \right|(|y|) |p_{d}(t,x+y)| dy
\\
&:= II''_{1}+II''_{2}+II''_{3}.
\end{aligned}
\end{equation*}
By \eqref{eqn 08.04.18:41}, and Theorem \ref{thm 07.25.16:53} (i) for any $n\leq m-1$ and $|y|= 2\theta_{a}(|x|)\geq 2|x|$ with $a\geq a_{d+2m}$,
\begin{align}\label{eqn 08.22.10:53}
|D^{m-1-n}_{x}p_{d}(t,x+y)|   
&\leq C \sum_{m-1-n-2l\geq0,l\in\bN_{0}} |x+y|^{m-1-n-2l} |p_{d+2(m-1-n-l)}(t,x+y)|  \nonumber
\\
&\leq C \sum_{m-1-n-2l\geq0,l\in\bN_{0}} |y|^{m-1-n-2l} |p_{d+2(m-1-n-l)}(t,x)|  \nonumber
\\
&\leq C 
  t \frac{K(\theta_{a}(|x|))}{\theta_{a}(|x|)^{d+m-n-1}}\exp{(C^{-1}th(\theta_{a}(|x|)))}. 
\end{align}
This, \eqref{e:H},  \eqref{eqn 05.12:15:07} and \eqref{Kellcomp}  yield
\begin{align*}
&\int_{|y| = 2\theta_{a}(|x|)} \left(\frac{d^{n}}{dr^{n}}j_{d} \right)(|y|) |D^{m-1-n}_{x}p_{d}(t,x+y)| dS  \nonumber
\\
&\leq C \Big( t \frac{K(\theta_{a}(|x|))}{\theta_{a}(|x|)^{d+m-n-1}}\exp{(-C^{-1}th(\theta_{a}(|x|)))}  \nonumber
 \times  \int_{|y| = 2\theta_{a}(|x|)}  |y|^{-d-n}\ell(|y|^{-1})  dS  \Big)  \nonumber
\\
&\leq C t \frac{K(\theta_{a}(|x|))}{\theta_{a}(|x|)^{d+m-n-1}}\exp{(-C^{-1}th(\theta_{a}(|x|)))}  
\times \theta(|x|)^{-n-1} \ell(\theta_{a}(|x|)^{-1}/2)    \nonumber
\\
&\leq C \ell(\theta_{a}(|x|)^{-1}/2) t \frac{K(\theta_{a}(|x|))}{\theta_{a}(|x|)^{d+m}}\exp{(-C^{-1}th(\theta_{a}(|x|)))}  \nonumber
\\ 
&\leq C K(2\theta_{a}(|x|))t \frac{K(\theta_{a}(|x|))}{\theta_{a}(|x|)^{d+m}}\exp{(-C^{-1}th(\theta_{a}(|x|)))}  \nonumber
\\
&\leq C th(2\theta_{a}(|x|)) \frac{K(\theta_{a}(|x|))}{\theta_{a}(|x|)^{d+m}}\exp{(-C^{-1}th(\theta_{a}(|x|)))}  
\leq C \frac{K(\theta_{a}(|x|))}{\theta_{a}(|x|)^{d+m}}\exp{(-C^{-1}th(\theta_{a}(|x|)))} 
\end{align*}
for any $a\geq a_{d,m+2}$. Hence, we have
$$
II''_{1} \leq C \frac{K(\theta_{a}(|x|))}{\theta_{a}(|x|)^{d+m}}\exp{(-C^{-1}th(\theta_{a}(|x|)))}.
$$
If $\theta_{a}(|x|)=|x|$, then by following \eqref{eqn 08.22.10:53} with the relation
$$
|x|/2 \leq |x+y| \leq 3|x| \quad \forall \, |y|=|x|/2,
$$ 
we can check that, for any $a\geq a_{d,m+2}$
\begin{align}\label{eqn 09.07.17:16}
II''_{2} \leq C\frac{K(\theta_{a}(|x|/2))}{\theta_{a}(|x|)^{d+m}}\exp{(-C^{-1}th(\theta_{a}(|x|)))} 
\leq C\frac{K(\theta(|x|))}{\theta_{a}(|x|)^{d+m}}\exp{(-C^{-1}th(\theta_{a}(|x|)))},
\end{align}
where for the last inequality, we used $\theta_{a}(|x|)/2 \leq \theta_{a}(|x|/2) \leq \theta_{a}(|x|)$ and the definition of $K$.
\\
If $\theta_{a}(|x|) = (\ell^{-1}(a/t))^{-1}$, then by Theorem \ref{thm 07.25.16:53} for any $a\geq a_{d,m+2}$ and for $|y|=\theta_{a}(|x|)/2$, we have
\begin{align*}
|D^{m-1-n}_{x}p_{d}(t,x+y)|  
&\leq C \sum_{m-1-n-2l\geq0,l\in\bN_{0}} |x+y|^{m-1-n-2l} |p_{d+2(m-1-n-2l)}(t,x+y)|  \nonumber
\\
&\leq C \sum_{m-1-n-2l\geq0,l\in\bN_{0}} |y|^{m-1-n-2l} |p_{d+2(m-1-n-2l)}(t,0)|  \nonumber
\\
&\leq C 
  t \frac{K(\theta_{a}(|x|))}{\theta_{a}(|x|)^{d+m-n-1}}\exp{(C^{-1}th(\theta_{a}(|x|)))},
\end{align*}
and thus \eqref{eqn 09.07.17:16} follows. Here we emphasize that the second inequality holds since $p_{d+2(m-1-n-2l)}$ are decreasing in $x$, and the the third inequality follows from Theorem \ref{thm 07.25.16:53}.

Finally, since $\theta_{a}(|x+y|) \leq 3\theta_{a}(|x|)$ (for $\theta_{a}(|x|)/2 \leq |y| \leq \theta_{a}(|x|)$),  using \eqref{e:H}, and Theorem \ref{thm 07.25.16:53} (i), for any $a\geq a_{d,m+2}$ we have
\begin{equation*}\label{eqn 05.04.15:38}
\begin{aligned}
II''_{3} &  \leq C \int_{\theta_{a}(|x|)/2\leq |y| \leq 2\theta_{a}(|x|)} |y|^{-d-m} \ell(|y|^{-1})|p_{d}(t,x+y)| dy
\\
& \leq C \frac{\ell(2(\theta_{a}(|x|))^{-1})}{\theta_{a}(|x|)^{d+m}} \int_{\theta_{a}(|x|)/2<|y|<2\theta_{a}(|x|)} t \frac{K(\theta_{a}(|x+y|))}{\theta_{a}(|x+y|)^{d}}e^{-C^{-1}th(\theta_{a}(|x+y|))} dy
\\
& \leq C \frac{\ell(2(\theta_{a}(|x|))^{-1})}{\theta_{a}(|x|)^{d+m}} e^{-C^{-1}th(3\theta(|x|))}
\\
& \hspace{30mm} \times \int_{\frac{\theta_{a}(|x|)}{2}<|y|<2\theta_{a}(|x|)} t \frac{K(\theta_{a}(|x+y|))}{\theta_{a}(|x+y|)^{d}}e^{-C^{-1}th(\theta_{a}(|x+y|))} dy
\\
&\leq C \frac{K(\theta_{a}(|x|)/2)}{\theta_{a}(|x|)^{d+m}} e^{-C^{-1}th(3\theta_{a}(|x|))}  
 \int_{\bR^{d}} t \frac{K(\theta(|y|))}{\theta(|y|)^{d}}e^{-C^{-1}th(\theta(|y|))} dy
\\
&\leq C \frac{K(\theta_{a}(|x|)/2)}{\theta_{a}(|x|)^{d+m}} e^{-tC^{-1} h(3\theta_{a}(|x|))} 
\leq C \frac{K(\theta_{a}(|x|))}{\theta_{a}(|x|)^{d+m}} e^{-tC^{-1} h(\theta_{a}(|x|))},
\end{aligned}
\end{equation*}
where the last two inequalities hold due to Lemma \ref{rmk 05.07.16:12} (ii), \eqref{eqn 05.12:15:07} and \eqref{eqn 08.23.11:50}. 

Therefore, for any $a\geq a_{d,m+2}$ we have
\begin{align*}
II'' \leq II''_{1}+II''_{2}+II''_{3} \leq C \frac{K(\theta_{a}(|x|))}{\theta_{a}(|x|)^{d+m}}\exp{(-C^{-1}th(\theta_{a}(|x|)))}.
\end{align*}
Hence, if we take $a$ sufficiently large enough so that $a \geq \tilde{C}b_{0}a_{d,m+2}$ (recall \eqref{eqn 09.07.15:24}), then we have
\begin{align*}
II &\leq II' + II'' \leq   C \frac{K(\theta_{a}(|x|))}{\theta_{a}(|x|)^{d+m}}\exp{(-C^{-1}th(\theta_{a}(|x|)))}.
\end{align*}
Combining this with  \eqref{eqn 05.04.12:07} and \eqref{eqn 07.27.13:48}, we have
\begin{align*}
|\mathcal{L}D^{m}_{x}p_{d}(t,x)| \leq |D^{m}_{x}p_{d}(t,x)|I + II + III
\leq  C \frac{K(\theta_{a}(|x|,t))}{[\theta_{a}(|x|,t)]^{d+m}}\exp{(-C^{-1}th(\theta_{a}(|x|,t)))}
\end{align*}
for all $a\geq \tilde{C}b_{0}a_{d,m+2}$. Hence, by taking $\alpha_{m}=\tilde{C}b_{0}a_{d,m+2}$($\geq a_{d,m+2}$), we prove the lemma.
\end{proof}

The following theorem is a generalization of previous lemma.
\begin{theorem}\label{thm 08.24.16:27}
Let $m,k\in \bN$, and let $\theta$ be given by \eqref{d:theta}. Suppose that $j_{d}$ satisfies $\bold{H}(d,m+2k)$ and $\ell$ satisfies the condition in Assumption \ref{ell_con} (ii).  
Then there exists $\alpha_{k,m}$ such that the following holds: for any $a\geq \alpha_{k,m}$ and $T>0$, there exists $C>0$ depending only on $\text{{\boldmath{$\kappa$}}}_{2k+m},d,m,a,k$ and $T$ such that
\begin{align*}
\left| \mathcal{L}^{k}D^{m}_{x}p_{d}(t,x) \right| & \leq  C t^{1-k} \frac{K(\theta_{a}(|x|,t))}{[\theta_{a}(|x|,t)]^{d+m}}\exp{(-C^{-1}th(\theta_{a}(|x|,t)))}
\end{align*}
for all $(t,x)\in (0,T]\times \bR^{d}$ and for all $a\geq \alpha_{k,m}$.
\end{theorem}
\begin{proof}
Due to Lemma \ref{pfracderivativeestimate},  the theorem holds for $k=1$ with $a\geq \alpha_{m}\geq a_{d,m+2}$. We can prove the theorem for general $k\geq 2$ by using an induction argument. 
\end{proof}

\begin{theorem}\label{cor 09.13.16:58}
Let $m,k\in \bN$. Suppose that $j_{d}$ satisfies $\bold{H}(d,m+2k)$, and $\ell$ satisfies Assumption \ref{ell_con} (i). Then for any $T<\infty$, we have
\begin{align}\label{eqn 09.13.15:42}
|\mathcal{L}^{k}D^{m}_{x}p_{d}(t,x)| \leq C t^{1-k}\frac{K(|x|)}{|x|^{d+m}}\exp{\left(-Cth(|x|)\right)} 
\end{align}
for all $(t,x)\in(0,T]\times(\bR^{d}\setminus\{0\})$, where the constant $C>0$ depends only on $\text{{\boldmath$\kappa$}}_{2k+m},d,m,k$ and $T$.
\end{theorem}

\begin{proof}
By \cite[Corollary 2.13]{cho21heat}, we have \eqref{eqn 09.13.15:42} when $k=m=0$. Using this with \eqref{eqn 08.04.18:41}, we also prove \eqref{eqn 09.13.15:42} for $k=0$. Finally, if we follow the argument in Theorem \ref{thm 08.24.16:27} with Remark \ref{rmk 09.13.15:47}, then we can prove \eqref{eqn 09.13.15:42}.
\end{proof}

The following theorem is a large time estimate of the heat kernel.

\begin{theorem}\label{thm 09.13.17:01}
For each $m,k\in \bN$, there exist $t_{k,m},C$ depending on $d,\text{{\boldmath$\kappa$}}_{2k+m},\ell,\text{{\boldmath$\delta$}},m,k$ such that for all $t\geq t_{k,m}$ and $x\in\bR^{d}$, we have
\begin{align}\label{eqn 09.23.13:23}
|\mathcal{L}^{k}D^{m}_{x}p_d(t,x)|  \le C\left( t^{-k} (h^{-1}(t^{-1}))^{-d-m} \mathbf{1}_{th(|x|)\geq1} +  t^{1-k}\frac{K(|x|)}{|x|^{d+m}}\mathbf{1}_{th(|x|)\leq 1} \right).
\end{align}
\end{theorem}

\begin{proof}
First, observe that
$$
\psi(|\xi|) \geq C_{0} h(|\xi|^{-1}) \geq C_{0} L(|\xi|^{-1}).
$$
Hence, if $|\xi| \geq 1$, then using \eqref{ll1}
$$
\psi(|\xi|) \geq C_{0} L(|\xi|^{-1}) = C_{0} \int_{|\xi|^{-1}}^{1} s^{-1}\ell(s^{-1}) ds \geq C_{0} \tilde{C} \log{|\xi|}.
$$
Take $\tilde{t}_{1}$ large enough so that $\tilde{C}C_{0}^{2}\tilde{t}_{1}/4  \geq 2d+m$. Then for all $t\geq \tilde{t}_{1}$, it follows that
\begin{align}\label{eqn 08.09.18:58}
\int_{|\xi|\geq1} |\xi|^{m} e^{-C_{0}t\psi(|\xi|)/4} d\xi \leq \int_{|\xi|\geq1} |\xi|^{m} e^{-(2d+m)\log{|\xi|}} d\xi 
\leq C \int_{|\xi|\geq1} |\xi|^{-2d} d\xi<\infty.
\end{align}

On the other hand, by \eqref{eqn 09.21.13:27} and \eqref{H:l} we obtain that
\begin{equation}\label{eqn 09.23.12:33}
c_2 \left(\frac{|\xi|}{(h^{-1}(t^{-1}))^{-1}}\right)^{\delta_{3}} \leq C th(|\xi|^{-1}) \quad\text{for}\;\; (h^{-1}(t^{-1}))^{-1}\leq |\xi| \leq 1. 
\end{equation}
Let $\tilde{t}_{0}$ be a large number so that $(h^{-1}(\tilde{t}^{-1}_{0}))^{-1} \leq 1$. Then for $t\geq (\tilde{t}_{0}\vee \tilde{t}_{1}):=t_{m}$, and $|\xi|\geq1$, applying \eqref{eqn 09.23.12:33} with $|\xi|=1$, we have
\begin{align}\label{eqn 09.23.12:34}
e^{-C_{0}th(|\xi|^{-1})/2}
\leq e^{-C_{0}t_{m}h(|\xi|^{-1})/4}e^{-C_{0}t h(1)/4} 
 \le e^{-C_{0}t_{1}h(|\xi|^{-1})/4}e^{-C^{-1}(h^{-1}(t^{-1}))^{\delta_{3}}}.  
\end{align}

Hence, for $t\geq (t_{0}\vee t_{1}):=t_{m}$, by the definition of $\mathcal{L}$, \eqref {eqn 7.20.1} and \eqref{eqn 05.27.15:40} we have
\begin{equation*}
\begin{aligned}
|\mathcal{L}^{k}D^{m}_{x}p_{d}(t,x)| \leq C \int_{\bR^{d}} t^{-k} (th(|\xi|^{-1}))^{k} |\xi|^{m} e^{-C_{0} t h(|\xi|^{-1})} d\xi 
\leq C \int_{\bR^{d}} t^{-k} |\xi|^{m} e^{-C_{0} t h(|\xi|^{-1})/2} d\xi .
\end{aligned}
\end{equation*}
Using \eqref{eqn 09.23.12:33} and \eqref{eqn 09.23.12:34}, we have
\begin{align*}
|\mathcal{L}^{k}D^{m}_{x}p_{d}(t,x)| 
 &\leq C  \int_{|\xi| \leq (h^{-1}(t^{-1}))^{-1}}  t^{-k} |\xi|^{m} d\xi  \nonumber 
 \\
 & \quad + C \int_{(h^{-1}(t^{-1}))^{-1} \leq |\xi| \leq 1} t^{-k}   |\xi|^{m} e^{-C^{-1}|\xi|^{\delta_{3}}(h^{-1}(t^{-1}))^{\delta_{3}}}d\xi  \nonumber
\\
& \quad + C \int_{|\xi|\geq1} t^{-k} |\xi|^{m}   e^{-C_{0}t_{1}h(|\xi|^{-1})/4}e^{-C^{-1}(h^{-1}(t^{-1}))^{\delta_{3}}} d\xi
\\
& := I_{1}+I_{2}+I_{3}.  \nonumber
\end{align*}
It is easy to see that $I_{1} \leq C t^{-k}(h^{-1}(t^{-1}))^{-d-m}$.  Using change of variables, we have 
$$
I_{2} \leq C t^{-k}(h^{-1}(t^{-1}))^{-d-m} \int_{1\leq |\xi|} |\xi|^{m} e^{-C^{-1}|\xi|^{\delta_{3}}} d\xi \leq Ct^{-k}(h^{-1}(t^{-1}))^{-d-m}.
$$
Also, by the relation $e^{-ax} \leq c(a,\gamma)x^{-\gamma}$ ($a,\gamma>0$ and $x>0$), and \eqref{eqn 08.09.18:58},  we have
$$
I_{3} \leq  C t^{-k} (h^{-1}(t^{-1}))^{-d-m} \int_{|\xi|\geq1} |\xi|^{m} e^{-C_{0}t_{1}\psi(|\xi|)/4} d\xi \leq Ct^{-k}(h^{-1}(t^{-1}))^{-d-m}.
$$
Therefore, we have
\begin{equation*}\label{eqn 07.12.16:00}
|\mathcal{L}^{k}D^{m}_{x}p_{d}(t,x)| \leq I_{1}+I_{2}+I_{3}
\leq C t^{-k} (h^{-1}(t^{-1}))^{-d-m}.
\end{equation*}

On the other hand, if $th(|x|) \leq 1$ (equivalently $|x|\geq h^{-1}(t^{-1})$), then we easily have
$$
t^{1-k}\frac{K(|x|)}{|x|^{d+m}} \leq t^{1-k} \frac{h(|x|)}{|x|^{d+m}} \leq t^{-k}(h^{-1}(t^{-1}))^{-d-m}.
$$
Hence, to prove the theorem, it only remains to show that 
\begin{equation}\label{eqn 08.24.16:29}
|\mathcal{L}^{k}D^{m}_{x}p_{d}(t,x)| \leq C t^{1-k} \frac{K(|x|)}{|x|^{d+m}}
\end{equation}
for $th(|x|) \leq 1$ and $t\geq t_{m}$. Indeed, proving \eqref{eqn 08.24.16:29} is very similar to the proof of Theorem \ref{thm 08.24.16:27} (with $\theta_{a}(|x|)=|x|$) using the estimation 
$$
|D^{m}_{x}p_{d}(t,x)| \leq Ct \frac{K(|x|)}{|x|^{d+m}} \quad \forall\, (t,x)\in(0,\infty)\times (\bR^{d}\setminus\{0\})
$$
which can be deduced from Proposition \ref{p:uhk_jump} and \eqref{eqn 08.04.18:41}. We remark that the only difference is to estimate the following term corresponding to $II''_{k_{0},3}$ in Theorem \ref{thm 08.24.16:27};
\begin{align*}
&\int_{|x|/2 \leq |y| \leq 2|x|} \left( \frac{d^{m}}{dr^{m}}j_{d} \right)(|y|) |\mathcal{L}^{k_{0}}p_{d}(t,x+y)| dy
\\
&\leq C \int_{|x|/2 \leq |y| \leq 2|x|} |y|^{-d-m}\ell(|y|^{-1})  t^{-k_{0}} (h^{-1}(t^{-1}))^{-d} \mathbf{1}_{th(|x+y|)\geq1} dy
\\
& \quad + C \int_{|x|/2\leq |y|\leq 2|x|} \mathbf{1}_{th(|x+y|)\leq1} |y|^{-d-m}\ell(|y|^{-1}) t^{1-k_{0}}\frac{K(|x+y|)}{|x+y|^{d}}  dy
\\
&\leq C  \frac{K(|x|)}{|x|^{d+m}} \int_{|x|/2 \leq |y| \leq 2|x|}   t^{-k_{0}} (h^{-1}(t^{-1}))^{-d} \mathbf{1}_{|x+y| \leq h^{-1}(t^{-1})} dy
\\
&\quad + C \frac{K(|x|)}{|x|^{d+m}} \int_{\bR^{d}} \mathbf{1}_{th(|y|)\leq1} t^{1-k_{0}}\frac{K(|y|)}{|y|^{d}}  dy
\\
&\leq  C \frac{K(|x|)}{|x|^{d+m}}  \left( \int_{|y|\leq h^{-1}(t^{-1})} t^{-k_{0}} (h^{-1}(t^{-1}))^{-d} dy  + \int_{|y|\geq h^{-1}(t^{-1})} t^{1-k_{0}}\frac{K(|y|)}{|y|^{d}} dy  \right) 
\\
&\leq Ct^{1-(k_{0}+1)} \frac{K(|x|)}{|x|^{d+m}},
\end{align*}
where the last inequality holds since
\begin{align}\label{eqn 08.07.16:31}
\int_{a}^{\infty} K(\rho) \rho^{-1} d \rho 
&= \int_{a}^{\infty} \rho^{-2} \int_{0}^{\rho} s\ell(s^{-1}) ds \rho^{-1} d\rho  \nonumber
\\
& = \int_{0}^{a}\int_{a}^{\infty} \rho^{-3} s\ell(s^{-1}) d\rho ds + \int_{a}^{\infty}\int_{s}^{\infty} \rho^{-3} s\ell(s^{-1}) d\rho ds  \nonumber
\\
&\leq  a^{-2} \int_{0}^{a} s\ell(s^{-1}) ds + \int_{a}^{\infty}s^{-1} \ell(s^{-1}) ds = h(a) \quad \forall \, a>0.
\end{align}
 Hence, for $t\geq t_{m}$, we have \eqref{eqn 09.23.13:23}. The theorem is proved.
\end{proof}

\begin{remark}\label{r:hku_largetime}
(i) If $\ell$ satisfies Assumption \ref{ell_con} (ii), then by Theorem \ref{thm 08.24.16:27}, we see that
\begin{align}\label{eqn 08.18.12:17}
|\mathcal{L}^{k}D^{m}_{x}p_d(t,x)| 
&\leq C t^{1-k} \frac{K(\theta_{a}(|x|,t))}{[\theta_{a}(|x|,t)]^{d+m}}\exp{(-C^{-1}th(\theta_{a}(|x|,t)))}   \nonumber
\\
&\leq C  t^{-k} \frac{1}{[\theta_{a}(|x|,t)]^{d+m}} 
\leq C t^{-k} \ell^{-1}(a/t)^{d+m} \quad \text{for}\;\;t\le T,
\end{align}
where $a\geq \alpha_{k,m}$ taken from Theorem \ref{thm 08.24.16:27}.

(ii) Suppose that the function $\ell$ satisfies Assumption \ref{ell_con} (ii)--(1). Then since we assume that $\ell \asymp \ell^{\ast}$,  there are $0<\beta<1<\lambda$ such that
$$
\beta \ell^{\ast}(r) \leq h(r^{-1}) \leq \lambda \ell^{\ast}(r) \quad \forall\,r>0.
$$
From this, we can get
$$
(h^{-1}(\beta r))^{-1} \leq \ell^{-1}(r) \leq (h^{-1}(\lambda r))^{-1} \quad \forall\, r>0.
$$
Hence,  by using   \eqref{eqn 08.18.12:17}, for fixed $T>0$, we can check that
\begin{align*}
|\mathcal{L}^{k}D^{m}_{x}p_{d}(t,x)| \leq C t^{-k} (h^{-1}(\lambda a/t))^{-d-m} = C t^{-k} (\tilde{h}^{-1}(t^{-1}))^{-d-m} \quad \text{for}\; t\leq T,
\end{align*}
where $\tilde{h}(r) = (a\lambda)^{-1}h(r)$. Using this and following the proof of Theorem \ref{cor 09.13.16:58}, if $\ell$ satisfies Assumption \ref{ell_con} (ii)--(1), then  \eqref{eqn 09.23.13:23} holds for all $t>0$ with $\tilde{h}:=(a\lambda)^{-1}h$ in place of $h$. Therefore, if $\ell$ satisfies Assumption \ref{ell_con} (ii)--(1), then we denote $\tilde{h}=h$ for convenience.

(iii) It seems nontrivial to control $(h^{-1}(a\lambda /t))^{-d-m}$ by using $(h^{-1}(t^{-1}))^{-d-m}$. Indeed, when $\delta_{1}=0$, by using \eqref{H:s} and $\ell (r^{-1}) \asymp h(r)$ only, we obtain
$$
(h^{-1}(a\lambda/t))^{-d-m} \leq  C_{0}(a,\lambda,d,m,T_{0}) (h^{-1}(t^{-1}))^{-d-m} \quad t>T_{0},
$$
and $C_{0}$ may blow up as $T_{0}\downarrow0$. 
\end{remark}

\mysection{Estimation of solution: Calder\'on-Zygmund approach}\label{sec 10.05.15:24}

In this section we prove some a priori estimates for solutions to the  equation with zero initial condition
\begin{equation}\label{mainequation-1}
\partial_{t}u = \mathcal{L}u + f,\quad t>0\,; \quad u(0,\cdot)=0.
\end{equation}
{\it Throughout this section, we suppose that $j_{d}$ satisfies $\bold{H}(d,4)$.} We first provide the representation formula. 

\begin{lemma} \label{u=qfsolution}

(i) Let $f\in C_c^\infty(\bR_+^{d+1})$ and define $u$ as 
\begin{equation} \label{u=qf}
\begin{gathered}
u(t,x)=\int_{0}^{t} \int_{\bR^{d}} p_{d}(t-s,x-y) f(s,y) dy ds.
\end{gathered}
\end{equation}
Then $u$ satisfies equation \eqref{mainequation-1} for each $(t,x)$.

(ii) Let $u\in C_c^\infty(\bR_+^{d+1})$ and denote $f:=\partial_{t}u-\mathcal{L}u$. Then $u$ and $f$ satisfies \eqref{u=qf}.
\end{lemma}

\begin{proof}
(i) By following Duhamel's principle for the heat equation (see e.g. \cite{evanspde}), we can easily prove the assertion.

(ii) Since $u\in \Ccinf(\bR^{d+1})$, we can check that $\partial_{t}u(t,x)$ and $\mathcal{L}u(t,x)$ are integrable in $x$. Hence, by taking the Fourier transform, we have
$$
\partial_{t}\cF_{d}(u)(t,\xi) = -\psi(|\xi|)\cF_{d}(u)(t,\xi) + \cF_{d}(f)(t,\xi)
$$
for any $(t,\xi)\in \bR^{d+1}_{+}$. Therefore, using the smoothness of $u$ again, we get
$$
\cF_{d}(u)(t,\xi) = \int_{0}^{t} \cF_{d}(f)(s,\xi) e^{-(t-s)\psi(|\xi|)}ds.
$$
Finally, by taking the inverse Fourier transform, Fubini's theorem, Parseval's identity, and \eqref{eqn 7.20.1}, we prove that $u$ satisfies \eqref{u=qf}. The lemma is proved.

\end{proof}

Throughout the rest of the article, we denote $\mathcal{L}p_{d}(t,x):= q(t,x)$.
\\
For $f\in \Ccinf(\bR^{d+1})$, we define 
\begin{align}
L_{0}f(t,x) := \int_{-\infty}^{t}\int_{\bR^{d}} p_{d}(t-s,x-y)f(s,y) dy ds,
\nn\\
Lf(t,x) := \int_{-\infty}^{t}\int_{\bR^{d}} p_{d}(t-s,x-y)\mathcal{L}f(s,y) dy ds,\label{def_Lf}
\end{align}
where $\mathcal{L}f(s,y):= \mathcal{L}(f(s,\cdot))(y)$.

Note that  
$\mathcal{L}f$ is bounded for any $f\in C_c^\infty (\bR^{d+1})$. Thus, the operator $L$ is well defined on $C_c^\infty (\bR^{d+1})$.
For each fixed $s$ and $t$ such that  $s<t$, define
\begin{align*}
\Lambda^{0}_{t,s}f(x) &:=\int_{\bR^d}p_{d}(t-s,x-y) \mathcal{L} f(s,y) dy,\\
\Lambda_{t,s}f(x) &:=\int_{\bR^d} \mathcal{L}p_{d}(t-s,x-y)  f(s,y) dy.
\end{align*}
If we use Theorem \ref{cor 09.13.16:58} and  Lemma \ref{rmk 05.07.16:12} (i) (if $\ell$ satisfies Assumption \ref{ell_con} (i)) or Theorem \ref{thm 08.24.16:27} and Lemma\ref{rmk 05.07.16:12} (ii) (if $\ell$ satisfies Assumption \ref{ell_con} (ii)), with Minkowski's inequality,  we can check that $\Lambda_{t,s}f$ and $\Lambda^{0}_{t,s}f$ are square integrable. Moreover, from the definition of $\mathcal{L}$, we have
\begin{equation*}
\begin{aligned}
\cF_{d}\{\mathcal{L}p_{d}(t-s,\cdot)\}(\xi) &= -\psi(|\xi|) \cF_{d}\{p_{d}(t-s,\cdot)\}(\xi).
\end{aligned}
\end{equation*}
Hence,
\begin{equation*}\label{eqn 06.12.15:17}
\begin{aligned}
\cF_d \{\Lambda^{0}_{t,s}f\}(\xi)  = -\psi(|\xi|)  \cF_{d}p_{d}(t-s,\xi) \hat{f}(s,\xi)
=\mathcal{F}_{d}\{\mathcal{L}p_{d}(t-s,\cdot)\}(\xi)\hat{f}(s,\xi)=\cF_d \{\Lambda_{t,s}f\}(\xi).
\end{aligned}
\end{equation*}
Thus, we have 
\begin{align}\label{eqn 06.02.15:23}
\mathcal{L}L_{0}f(t,x) = Lf(t,x)  \nonumber
&= \lim_{\varepsilon \downarrow 0} \int_{-\infty}^{t-\varepsilon}\left(\int_{\bR^{d}} \mathcal{L}p_{d}(t-s,x-y) f(s,y) dy \right) ds  
\\
&= \lim_{\varepsilon \downarrow 0} \int_{-\infty}^{t-\varepsilon}\left(\int_{\bR^{d}}q(t-s,x-y) f(s,y) dy \right) ds.
\end{align}

\begin{remark}\label{rmk 08.31.14:17}
Our approach to obtain estimation \eqref{eqn 05.27.14:12} is based on $L_{p}$-boundedness of linear operators. Shortly, we will prove $L_{2}$-boundedness of the operator $L$ in Lemma \ref{22estimate}, and then we will use Lemma \ref{ininestimate} ~ Lemma \ref{outinspaceestimate} to get BMO-$L_{\infty}$ estimation (Theorem \ref{thm 05.14.18:28} (i)) of the operator $L$. Finally, using properties of $q$, we prove the main result (Theorem \ref{thm 05.14.18:28} (ii)) in this section. For this reason, we need estimation of $\partial_{t}q = \mathcal{L}^{2}p_{d}$. In Section \ref{sec3}, we derived estimation of $\mathcal{L}^{2}p_{d}$ under the condition that $j_{d}$ satisfies Assumption \ref{asm 07.10.16:46} 4-times. Hence, to get estimations of solutions, we need 4-times differentiability on $j_{d}$. 

\end{remark}

For $(t,x)\in\R^{d+1}$ and $b>0$,  denote
\begin{equation}\label{eqn 09.07.16:22}
 Q_b(t,x)=(t-(h(b))^{-1},\,t+(h(b))^{-1})\times {B}_b(x),
\end{equation}
and
\begin{equation*}
Q_b=Q_b(0,0), \quad B_b=B_b(0).
\end{equation*}
If $\ell$ satisfies Assumption \ref{ell_con} (ii)--(1), then we define $Q_{b}$ by using 
$$
\tilde{h}(b) = (a\lambda)^{-1}h(b), \quad \lambda>1, \quad a\geq \alpha_{2,0}\vee \alpha_{1,1}
$$ 
(see Remark \ref{r:hku_largetime} (ii)) in place of $h$. However, for notational simplicity, we still denote $\tilde{h}$ by $h$ in \eqref{eqn 09.07.16:22}. 

For  measurable subsets $Q\subset \R^{d+1}$ with  finite measure and  locally integrable functions $f$, define
\begin{equation*}
f_Q=\aint_{Q}f(s,y)dyds=\frac{1}{|Q|}\int_{Q}f(s,y)dyds,
\end{equation*}
where $|Q|$ is the Lebesgue measure of $Q$.

Recall that we assume that Assumption \ref{ass bernstein} and Assumption \ref{ell_con}. We will use the following lemma in the rest of this section. 
\begin{lemma}\label{l:qest_case1} 
There exist $C_{1}=C_{1}(d,\text{\boldmath{$\kappa$}}_{3},\text{\boldmath{$\delta$}},\ell)$ and $C_{2}=C_{2}(d,\text{\boldmath{$\kappa$}}_{2},\text{\boldmath{$\delta$}},\ell)$ such that for any $b>0$
\begin{align}
&\int_{(h(b))^{-1}}^\infty \int_{|y|\ge b } |D_xq(s,y)| dy ds\le C_{1}b^{-1},\label{q1upper}\\
&\int_{(h(4b))^{-1}}^{\infty} \int_{|y|\leq 4b} |q(s,y)| dy ds\le C_{2}.\label{q0upper}
\end{align}
\end{lemma}

\begin{proof}
 We will consider two cases separately.

\noindent{\bf (Case 1)} $\ell$ satisfies Assumption \ref{ell_con} (i).

In this case, \eqref{q0upper} is equal to \eqref{eqn 09.14.11:27} whose proof is contained in Lemma \ref{l:qest_case1_pf}. Hence, we only consider \eqref{q1upper}. Observe that, if we take $t_{1,1}$ from Theorem \ref{thm 09.13.17:01}
\begin{equation}\label{eqn 09.06.19:51}
\begin{aligned}
\int_{(h(b))^{-1}}^\infty \int_{|y| \geq b} |D_xq(s,y)| dy ds
&\leq  \int_{(h(b))^{-1}}^\infty \int_{b\leq |y| \leq h^{-1}(s^{-1})} \cdots dy ds
+\int_{(h(b))^{-1}}^\infty \int_{|y| \geq h^{-1}(s^{-1})}   \cdots dy ds
\\
&=:I_1+I_2.
\end{aligned}
\end{equation}
By Lemma \ref{l:qest_case1_pf}, we obtain $I_1 \leq Cb^{-1}$.  
By Theorem \ref{cor 09.13.16:58} and Theorem \ref{thm 09.13.17:01},
$$
I_2   =\int_{(h(b))^{-1}}^{\infty}\int_{|y| \geq h^{-1}(s^{-1})} \frac{K(|y|)}{|y|^{d+1}} \left( e^{-Cth(|y|)} \mathbf{1}_{s \leq t_{1,1} } + \mathbf{1}_{s \geq t_{1,1} } \right)  dy ds.
$$
Hence, using \eqref{eqn 08.07.16:31}, we have
\begin{equation}\label{eqn 08.30.11:05}
\begin{aligned}
I_2  & =\int_{(h(b))^{-1}}^{\infty}\int_{|y| \geq h^{-1}(s^{-1})} \frac{K(|y|)}{|y|^{d+1}}dy ds
 \leq C\int_{(h(b))^{-1}}^{\infty} h^{-1}(s^{-1})^{-1}\int_{|y| \geq h^{-1}(s^{-1})} \frac{K(|y|)}{|y|^{d}} dy ds
\\
& \leq C\int_{(h(b))^{-1}}^{\infty} h^{-1}(s^{-1})^{-1}s^{-1}ds.
\end{aligned}
\end{equation}
Thus, by Lemma \ref{l:qest_offd} with $f(r)=h(r^{-1})$ and $k=1$, we obtain $I_{2} \le Cb^{-1}$.

\noindent{\bf (Case 2)} $\ell$ satisfies the Assumption \ref{ell_con} (ii).

First, we consider \eqref{q1upper}. Observe that 
\begin{equation*}
\begin{aligned}
\int_{(h(b))^{-1}}^\infty \int_{|y| \geq b} |D_xq(s,y)| dy ds
&\leq  \int_{(h(b))^{-1}}^\infty \int_{b\leq |y| \leq h^{-1}(s^{-1})}   \cdots  dy ds
 +\int_{(h(b))^{-1}}^\infty \int_{|y| \geq h^{-1}(s^{-1})} \mathbf{1}_{s \leq t_{1,1} }  \cdots dy ds
\\
&\quad +\int_{(h(b))^{-1}}^\infty \int_{|y| \geq h^{-1}(s^{-1})} \mathbf{1}_{s \geq t_{1,1} } \cdots  dy ds
\\
&=:I_4+I_5 + I_{6},
\end{aligned}
\end{equation*}
where $t_{1,1}$ is taken from Theorem \ref{thm 09.13.17:01}.

Suppose that $\ell$ satisfies Assumption \ref{ell_con} (ii)--(2).  Like \eqref{eqn 09.06.19:51}
 Using \eqref{q1upper_case2_pf}, we have $I_{4}\leq Cb^{-1}$. Also, by Theorem \ref{thm 08.24.16:27} and Lemma \ref{rmk 05.07.16:12} (ii)
\begin{align*}
I_{5} &\leq  C \int_{(h(b))^{-1}}^{\infty} \int_{|y| \geq h^{-1}(s^{-1})} \frac{K(\theta_{a}(|y|,s))}{\theta_{a}(|y|,s)^{d+1}} dy ds
\\
&\leq  C \int_{(h(b))^{-1}}^{\infty} s^{-1} \int_{|y| \geq h^{-1}(s^{-1})} s \frac{K(\theta_{a}(|y|,s))}{\theta_{a}(|y|,s)^{d}}|y|^{-1} dy ds
\\
&\leq C \int_{(h(b))^{-1}}^{\infty} h^{-1}(s^{-1})^{-1} s^{-1} \int_{|y| \geq h^{-1}(s^{-1})} s \frac{K(\theta_{a}(|y|,s))}{\theta_{a}(|y|,s)^{d}} dy ds
\\
&\leq C  \int_{(h(b))^{-1}}^{\infty} h^{-1}(s^{-1})^{-1} s^{-1}ds \leq Cb^{-1},
\end{align*}
where $a\geq \alpha_{1,1}$ comes from Theorem \ref{thm 08.24.16:27} and we used Lemma \ref{l:qest_offd} with $f(r)=h(r^{-1})$ and $k=1$ to obtain the last inequality. Moreover, by following the argument in \eqref{eqn 08.30.11:05}, we have $I_{6} \leq Cb^{-1}$. Hence we have \eqref{q1upper} when $\ell$ satisfies Assumption \ref{ell_con} (ii)--(2).

Suppose that $\ell$ satisfies Assumption \ref{ell_con} (ii)--(1). Then, due to Remark \ref{r:hku_largetime} (ii) and Lemma \ref{l:qest_offd} with $f(r)=h(r^{-1})$ and $k=1$ we have
\begin{align*}
I_{4}
\leq C \int_{(h(b))^{-1}}^{\infty}\int_{b\leq |y| \leq h^{-1}(s^{-1})} s^{-1} (h^{-1}(s^{-1}))^{-d-1} dy ds
\leq C\int_{(h(b))^{-1}}^{\infty}s^{-1}(h^{-1}(s^{-1}))^{-1} ds \leq Cb^{-1}.
\end{align*}
Also, since $I_{5}$ and $I_{6}$ can be handled in the exactly same way as above, we again have \eqref{q1upper}.

Now, consider \eqref{q0upper}. If $\ell$ satisfies Assumption \ref{ell_con} (ii)--(2), then by Lemma \ref{lem 09.21.12:37} we already have the desired result. Hence, we only need to consider the case where $\ell$ satisfies Assumption \ref{ell_con} (ii)--(1). Again by Remark \ref{r:hku_largetime} (ii)
\begin{align*}
\int_{(h(4b))^{-1}}^\infty \int_{|y|\leq 4b} |q(s,y)| dy ds 
&\leq C \int_{(h(4b))^{-1}}^\infty \int_{0}^{4b} s^{-1}(h^{-1}(s^{-1}))^{-d}\mathbf{1}_{sh(\rho) \geq 1}  \rho^{d-1}d\rho ds   \nonumber
\\
&\leq C \int_{(h(4b))^{-1}}^\infty \int_{0}^{4b}  s^{-1}(h^{-1}(s^{-1}))^{-d} \rho^{d-1}d\rho ds
\\
&\leq Cb^{d} \int_{(h(4b))^{-1}}^{\infty}  s^{-1}(h^{-1}(s^{-1}))^{-d} ds \leq C,
\end{align*}
where we used Lemma \ref{l:qest_offd} with $f(r)=h(r^{-1})$. The lemma is proved.
\end{proof}

Recall that the operator $L$ defined in \eqref{def_Lf} satisfies \eqref{eqn 06.02.15:23}.
 
\begin{lemma} \label{22estimate}
For any $f\in\Ccinf(\bR^{d+1})$, we have
\begin{equation*}
\begin{gathered}
\|L f\|_{L_2(\bR^{d+1})}\leq C(d) \|f\|_{L_2(\bR^{d+1})}.
\end{gathered}
\end{equation*}
Consequently, the  operators $L$ can be  continuously extended to $L_2(\bR^{d+1})$. 
\end{lemma}

\begin{proof}
Using \eqref{eqn 7.20.1}, and the definition of $\mathcal{L}$, we have
\begin{align*}
|\mathcal{F}_{d+1}(q)(\tau,\xi)| &=\left| C \int_{\bR} \psi(|\xi|) e^{-t\psi(|\xi|)} e^{-it \tau} dt  \right| 
\leq  C \int_{\bR} \psi(|\xi|) e^{-t\psi(|\xi|)}  dt   \leq C.
\end{align*}
This implies that the operator $L$ has a bounded Fourier multiplier. From this, we directly get the desired result.
\end{proof}

\begin{lemma} \label{ininestimate}
Let $b>0$ and  $f\in C^{\infty}_c(\bR^{d+1})$  have a support in   $ (-3(h(b))^{-1}, 3(h(b))^{-1})\times B_{3b}$. Then,
\begin{equation*}
\begin{gathered}
\aint_{Q_b}|L f (t,x)|dxdt \leq C(d) \|f\|_{L_\infty(\bR^{d+1})}.
\end{gathered}
\end{equation*}
\end{lemma}

\begin{proof}
By H\"older's inequality and Lemma \ref{22estimate},
\begin{align*}
\aint_{Q_b}|L f (t,x)|dxdt \leq &  \;\|L f\|_{L_2(\bR^{d+1})}|Q_b|^{-1/2}
\leq  \;C \|f\|_{L_2(\bR^{d+1})}|Q_b|^{-1/2}
\\
= & \;\left( \int_{-3(h(b))^{-1}}^{3(h(b))^{-1}} \int_{B_{3b}} |f(t,x)|^{2} dydt \right)^{1/2}|Q_b|^{-1/2}
\leq  C \|f\|_{L_\infty(\bR^{d+1})}.
\end{align*}
The lemma is proved.
\end{proof}

\begin{lemma} \label{inwholeestimate} 
Let $b>0$, and $f\in C_c^\infty(\bR^{d+1})$  have a support in $(-3(h(b))^{-1}, \infty)\times \bR^d$. Then,
\begin{equation*}
\begin{gathered}
\aint_{Q_b}|L f (t,x)|dxdt \leq C(d,\text{\boldmath{$\kappa$}}_{2},\text{\boldmath{$\delta$}},\ell) \|f\|_{L_\infty(\bR^{d+1})}.
\end{gathered}
\end{equation*}
\end{lemma}

\begin{proof}
Take a function $\eta \in C^{\infty}(\bR)$ such that $0\leq \eta \leq 1$, $\eta=1$ on $(-\infty,2(h(b))^{-1})$ $\eta=0$ on $(3(h(b))^{-1},\infty)$. Then for $(t,x) \in Q_{b}$ we can check that $Lf (t,x) = L(f\eta)(t.x)$. Moreover, $|f\eta (t,x)| \leq |f (t,x)|$ for all $(t,x)$. Therefore, we may assume that the support of $f$ lies in $(-3(h(b))^{-1},3(h(b))^{-1})\times \bR^{d}$. Also, if we take $\zeta=\zeta(x) \in C_c^\infty (\bR^d)$ such that $\zeta=1$ in $B_{7b/3}$, $\zeta=0$ outside of $B_{8b/3}$ and $0\leq\zeta\leq1$, then due to the linearity of $L$ we have $L f = L (f\zeta) + L (f(1-\zeta))$, where $L(f\zeta)$ can be handled by Lemma \ref{ininestimate}. Hence, we may further assume that $f(t,z)=0$ for $|z|<2b$. 

Under the above setting, for any $x\in B_b$,
\begin{align*}
\int_{\bR^d} \left|q (t-s,x-y) f(s,y)\right| dy =\int_{|y-x|\geq 2b} \left|q (t-s,y) f(s,x-y)\right| dy 
\leq  \int_{|y|\geq b} \left|q (t-s,y) f(s,x-y) \right| dy.
\end{align*}
Using Theorem \ref{cor 09.13.16:58} (when $\ell$ satisfies Assumption \ref{ell_con} (i)), Theorem \ref{thm 09.13.17:01} (when $\ell$ satisfies Assumption \ref{ell_con} (ii)),  and \eqref{eqn 08.07.16:31},
\begin{align*}
 \int_{|y|\geq b} \left|q (t-s,y) f(s,x-y) \right| dy 
\leq C\|f\|_{L_\infty(\bR^{d+1})}\mathbf{1}_{|s|\leq 3(h(b))^{-1}} \int_b^\infty  \frac{K(\rho)}{\rho^{d}}\rho^{d-1}d\rho
\leq C\|f\|_{L_\infty(\bR^{d+1})} \mathbf{1}_{|s|\leq 3(h(b))^{-1}}   h(b).
\end{align*}
Note that if $|t| \leq (h(b))^{-1}$ and $|s|\leq 3(h(b))^{-1}$ then $|t-s|\leq 4(h(b))^{-1}$. Thus we have
\begin{align*}
|L f(t,x)| &\leq C\|f\|_{L_\infty(\bR^{d+1})} h(b) \int_{|t-s|\leq 4(h(b))^{-1}} ds
 \leq C \|f\|_{L_{\infty}(\R^{d+1})}.
\end{align*}
This implies the desired estimate. The lemma is proved.
\end{proof}

\begin{lemma}
\label{outwholetimeestimate}
Let $b>0$, and $f\in C_c^\infty(\bR^{d+1})$  have a support in  $(-\infty, -2(h(b))^{-1})\times \bR^d$. Then there is $C=C(d,\text{\boldmath{$\kappa$}}_{4},\text{\boldmath{$\delta$}},\ell)$ such that for  any $(t_1,x), (t_2,x)\in Q_b$, 
\begin{gather*}
\aint_{Q_b}\aint_{Q_b} |L f (t_1,x)-L f(t_2,x)| dx dt_1 d\tilde{x} dt_2\leq C \|f\|_{L_\infty(\bR^{d+1})}.
\end{gather*}

 \end{lemma}

\begin{proof}
We will show that 
$$
|L f (t_1,x)-L f(t_2,x)|   \leq C \|f\|_{L_\infty(\bR^{d+1})},
$$
and this certainly proves the lemma. 

Without loss of generality, we assume that $t_{1}>t_{2}$. Recall $f(s,y)=0$ if $s\geq -2(h(b))^{-1}$. Thus,  if $t>-(h(b))^{-1}$,  by  applying the fundamental theorem of calculus and the relation $\partial_{t}q = \partial_{t}\mathcal{L}p_{d}=\mathcal{L}^{2}p_{d}$, we have
\begin{align*}
|L f(t_1,x)-L f(t_2,x)|
&= \left|\int_{-\infty}^{-2(h(b))^{-1}}  \int_{\bR^d}  \int_{t_2}^{t_1} \mathcal{L}^{2}p_{d}(t-s,x-y) f(s,y)  dt dy ds \right|.
\end{align*}
Suppose that $\ell$ satisfies Assumption \ref{ell_con} (i). Observe that for $k\in \bN$,
\begin{align*}
\int_{\bR^{d}} |\mathcal{L}^{k}p_{d}(s,y)| dy  \nonumber
&\leq C \int_{\bR^{d}}  s^{1-k} \frac{K(|y|)}{|y|^{d}} e^{-Csh(|x|)} dy  \nonumber
\\
&\quad  +  C\int_{\bR^{d}}  \left( s^{-k} (h^{-1}(s^{-1}))^{-d} \mathbf{1}_{sh(|y|)\geq1} +  s^{1-k}\frac{K(|y|)}{|y|^{d}}\mathbf{1}_{sh(|y|)\leq 1} \right)dy  \nonumber
\\
&:= C(I_{k,1}(s)+I_{k,2}(s))
\end{align*}
due to  Theorem \ref{cor 09.13.16:58} and Theorem \ref{thm 09.13.17:01}. By Lemma \ref{rmk 05.07.16:12} (i), we have $I_{k,1}(s)\leq Cs^{-k}$. For $I_{k,2}(s)$, observe that
\begin{align*}
I_{k,2}(s) &=  \left(\int_{|y|\leq h^{-1}(s^{-1})}s^{-k}(h^{-1}(s^{-1})^{-d} dy + \int_{h^{-1}(s^{-1})}^{\infty}s^{1-k}K(\rho)\rho^{-1} d\rho \right)
\leq Cs^{-k}
\end{align*}
due to \eqref{eqn 08.07.16:31}. Hence, we have
\begin{align}\label{eqn 09.23.16:51}
\int_{\bR^{d}} |\mathcal{L}^{k}p_{d}(s,y)|dy \leq C(I_{k,1}(s)+I_{k,2}(s)) \leq Cs^{-k}.
\end{align} Also, one can check the same result holds when $\ell$ satisfies Assumption \ref{ell_con} (ii). Therefore, 
\begin{align*}
&\int_{\bR^d} \int_{t_2}^{t_1} |\mathcal{L}^{2}p_{d} (t-s,x-y) f(s,y)| dt dy
\\
&\leq C \|f\|_{L_{\infty}} \int_{t_{2}}^{t_{1}} (I_{2,1}(t-s)+I_{2,2}(t-s))  dt
\leq C \|f\|_{L_\infty(\bR^{d+1})} \int_{t_2}^{t_1} (t-s)^{-2}dt.
\end{align*}
Thus,  if  $-(h(b))^{-1}\leq t_2<t_1 \leq (h(b))^{-1}$, 
\begin{align*}
\Big|\int_{-\infty}^{-2(h(b))^{-1}}  \int_{\bR^d}  \int_{t_2}^{t_1}\mathcal{L}^{2}p_{d} (t-s,x-y) f(s,y)  dt dy ds\Big|
&\leq C \|f\|_{L_\infty(\bR^{d+1})} \left(\int_{t_2}^{t_1}\int_{-\infty}^{-2(h(b))^{-1}} (t-s)^{-2}ds dt \right)
\\
&\leq C \|f\|_{L_\infty(\bR^{d+1})} \left(\int_{t_2}^{t_1} h(b) dt \right)
\leq C \|f\|_{L_\infty(\bR^{d+1})}.
\end{align*}
This completes the proof.
\end{proof}

\begin{lemma}
\label{outoutspaceestimate}
Let $b>0$, and $f\in C_c^\infty(\bR^{d+1})$ have a support in  $(-\infty,-2(h(b))^{-1})\times B_{2b}^c$. Then there is $C=C(\alpha,d,\text{\boldmath{$\kappa$}}_{3},\text{\boldmath{$\delta$}},\ell)$ such that for any $(t,x_{1}),(t,x_{2})\in Q_b$, 
$$
\aint_{Q_{b}}\aint_{Q_b} |L f (t,x_{1}) - L f(t,x_{2})| dx_{1} dt dx_{2} d\tilde{t} \leq C \|f\|_{L_\infty(\bR^{d+1})}.
$$
\end{lemma}

\begin{proof}
Like the previous lemma, it is enough to show that
\begin{gather*}
|L f (t,x_{1}) - L f(t,x_{2})| \leq C \|f\|_{L_\infty(\bR^{d+1})}.
\end{gather*}

Recall $f(s,y)=0$ if $s\geq -2(h(b))^{-1}$ or $|y|\leq 2b$. Thus,  if $t>-(h(b))^{-1}$, by the fundamental theorem of calculus,
\begin{equation*}
\begin{aligned}
|L f(t,x_{1}) - L f(t,x_{2})|
&=\Big|\int_{-\infty}^{-2(h(b))^{-1}}  \int_{|y|\geq 2b} \left(q (t-s,x_{1}-y) -q(t-s,x_{2}-y) \right) f(s,y) dy ds \Big|
\\
&= \Big|\int_{-\infty}^{-2(h(b))^{-1}}  \int_{|y|\geq 2b} \int_{0}^{1}  \nabla_{x} q(t-s,\lambda(x_{1},x_{2},u)-y) \cdot (x_{1}-x_{2}) du f(s,y) dy ds \Big|,
\end{aligned}
\end{equation*}
where $\lambda(x_{1},x_{2},u)=ux_{1}+(1-u)x_{2}$. Since $x_{1},x_{2}\in B_{b}$ and $|y|\geq 2b$, we can check that $|\lambda(x_{1},x_{2},u)-y| \geq b$. Thus, by the change of variable $(\lambda(x_{1},x_{2},u)-y) \to y$,
\begin{equation*}
\begin{aligned}
|L f(t,x_{1}) - L f(t,x_{2})|
& \leq C  b \|f\|_{L_\infty(\bR^{d+1})}  \int_{-\infty}^{ -2(h(b))^{-1}} \int_{|y|\geq b} |\nabla_{x}q (t-s,y)| dy ds
\\
& \leq C b  \|f\|_{L_\infty(\bR^{d+1})}  \int_{(h(b))^{-1}}^{\infty} \int_{|y|\geq b} |\nabla_{x}q (s,y)| dy ds.
\end{aligned}
\end{equation*}
Thus, by \eqref{q1upper}, we obtain the desired result.
\end{proof}

\begin{lemma} \label{outinspaceestimate}
Let $b>0$, and $f\in C_c^\infty(\bR^{d+1})$  have a support in   $(-\infty,-2(h(b))^{-1})\times B_{3b}$. Then for any $(t,x)\in Q_b$
\begin{gather*}
\aint_{Q_b}|L f (t,x)|dx dt \leq C(d,\text{\boldmath{$\kappa$}}_{2},\text{\boldmath{$\delta$}},\ell)\|f\|_{L_\infty(\bR^{d+1})}.
\end{gather*} 
\end{lemma}

\begin{proof}
For $(t,x)\in Q_b$,
\begin{align*}
|L f(t,x)| \leq&\int_{-\infty}^{-2(h(b))^{-1}} \int_{B_{3b}} |q (t-s,x-y)f(s,y)|dyds
\\
\leq& C \|f\|_{L_\infty} \int_{-\infty}^{-2(h(b))^{-1}} \int_{B_{3b}} |q (t-s,x-y)|dyds
\\
\leq& C \|f\|_{L_\infty} \int_{(h(b))^{-1}}^{\infty} \int_{B_{4b}} |q(s,y)|dyds := C \|f\|_{L_\infty} \left( I +II \right),
\end{align*}
where
$I=\int^{(h(4b))^{-1}}_{(h(b))^{-1}} \int_{B_{4b}} |q (s,y)|dyds$, and $II=\int_{(h(4b))^{-1}}^{\infty} \int_{B_{4b}} |q (s,y)|dyds$.
\\
\noindent
By \eqref{eqn 09.23.16:51} with $k=1$ and \eqref{eqn 05.12:15:07}
\begin{align*}
I &= \int^{(h(4b))^{-1}}_{(h(b))^{-1}} \int_{B_{4b}}  |q (s,y)|dyds
\leq C \int^{(h(4b))^{-1}}_{(h(b))^{-1}} s^{-1}ds
 \leq C \log\left(\frac{(h(4b))^{-1}}{(h(b))^{-1}} \right) 
\leq C \log\left(16\right).
\end{align*}
By \eqref{q0upper}, we also have $II\le C$.
Thus, $I$ and $II$ are bounded by a constant independent of $b$. Hence, we have the desired result. The lemma is proved.
\end{proof}

\begin{corollary}
\label{outwholeestimate}
There is $C=C(d,\text{\boldmath{$\kappa$}}_{4},\text{\boldmath{$\delta$}},\ell)$ such that for any $f\in \Ccinf(\bR^{d+1})$ and $b>0$
\begin{align}\label{eqn 09.28.15:28}
\aint_{Q_b}\aint_{Q_b}|L f (t,x)-L f(s,y)|dxdtdyds \leq C \|f\|_{L_\infty(\bR^{d+1})}.
\end{align}
\end{corollary}

\begin{proof}
Take functions $\eta=\eta(t) \in C^\infty(\bR)$ and $\zeta=\zeta(x)\in \Ccinf(\bR^{d})$ as follows;
\begin{gather*}
0\leq \eta \leq 1, \quad \eta=1 \quad \text{on} (-\infty,-8(h(b))^{-1}/3) \quad \eta(t)=0 \quad \text{for} \quad t\geq -7(h(b))^{-1}/3,
\\
0\leq \zeta \leq 1, \quad \zeta=1 \quad \text{on} \quad B_{7b/3} \quad \zeta=0 \quad \text{on} \quad B_{8b/3}.
\end{gather*} 
Then for any $(t,x),(s,y)$ using linearity of $L$, we have
\begin{equation*}
\begin{aligned}
|Lf(t,x)-Lf(s,y)|
&\leq |Lf(t,x) - Lf(s,x)| + |Lf(s,x) - Lf(s,y)| 
\\
&\leq \left( |Lf_{1} (t,x) - Lf_{1} (s,x)| + |Lf_{1}(s,x) - Lf_{1}(s,y)| \right)   +  |Lf_{2} (t,x) - Lf_{2} (s,x)| 
\\
& \quad + |Lf_{3}(s,x) - Lf_{3}(s,y)|  + |Lf_{4}(s,x) - Lf_{4}(s,y)|
\\
&:= \sum_{i=1}^{5} A_{i}(t,s,x,y),
\end{aligned}
\end{equation*}
where $f_{1}=f(1-\eta)$, $f_{2}=f\eta$ $f_{3}=f_{2}(1-\zeta)=f\eta(1-\zeta)$ and $f_{4}=f_{2}\zeta=f\eta\zeta$. If we use Lemma \ref{inwholeestimate} to $A_{1}+A_{2}$, Lemma \ref{outwholetimeestimate} to $A_{3}$, and then, Lemma \ref{outoutspaceestimate}, and Lemma \ref{outinspaceestimate} to $A_{4}+A_{5}$, we have the desired result. The corollary is proved.
\end{proof}

For  locally integrable functions $f$ on $\bR^{d+1}$, we define the BMO semi-norm of $f$ on $\bR^{d+1}$ as
\begin{equation*}
\|f\|_{BMO(\bR^{d+1})}=\sup_{Q\in \bQ} \aint_Q |f(t,x)-f_Q| dtdx
\end{equation*}
where $f_Q=\aaint_Q f(t,x)dtdx$ and
\begin{equation*}
\bQ:=\{Q_b(t_0,x_0) : b>0, (t_0,x_0)\in \bR^{d+1} \}.
\end{equation*}

For  measurable functions $f$ on $\bR^{d+1}$, we define the sharp function
\begin{align*}
f^{\#}(t,x)=\sup_{Q_{b}(r,z)\in \bQ} \aint_{Q_b(r,z)}|f(s,y)-f_{Q_b(r,z)}|dsdy,
\end{align*}
where the supremum is taken over all $Q_{b}(r,z)\in\bQ$ containing $(t,x)$.

\begin{theorem}[Fefferman-Stein Theorem]
\label{feffermanstein}
For any $1<p<\infty$, and $f\in L_p(\bR^{d+1})$,
\begin{equation*}
C^{-1}\|f^{\#}\|_{L_p(\bR^{d+1})}\leq \|f\|_{L_p(\bR^{d+1})}\leq C\|f^{\#}\|_{L_p(\bR^{d+1})},
\end{equation*}
where $C>1$ depends on $d,p,\text{{\boldmath$\kappa$}}_{0}$.
\end{theorem}

\begin{proof}
See  \cite[Theorem I.3.1, Theorem IV.2.2]{stein1993harmonic}.  We only remark that  due to \eqref{eqn 05.12:15:07},  the balls $Q_b(s,y)$ satisfy the conditions (i)--(iv) in  \cite[Section 1.1] {stein1993harmonic}.
\end{proof}

Recall that the linear operator $L$ is given by
$$
L f(t,x)=\lim_{\varepsilon \downarrow 0} \int_{-\infty}^{t-\varepsilon} \left(\int_{\bR^d} q(t-s,x-y) f(s,y) dy\right) ds.
$$
The following theorem is our main result in this section. The proof is quite standard.

\begin{theorem}\label{thm 05.14.18:28}

(i) For any $f\in L_2(\bR^{d+1})\cap L_\infty(\bR^{d+1})$,
\begin{equation*}
\begin{gathered}
\|L f\|_{BMO(\bR^{d+1})} \leq C(d,\text{\boldmath{$\kappa$}}_{4},\text{\boldmath{$\delta$}},\ell) \|f\|_{L_\infty(\bR^{d+1})}.
\end{gathered}
\end{equation*}

(ii) For any $p,q\in(1,\infty)$ and $f\in C_c^\infty (\bR^{d+1})$,
\begin{equation}\label{qpestimate}
\begin{gathered}
\|L f\|_{L_q(\bR;L_p(\R^{d}))} \leq C(d,p,q,\text{\boldmath{$\kappa$}}_{4},\text{\boldmath{$\delta$}},\ell) \|f\|_{L_q(\bR;L_p(\R^{d}))}.
\end{gathered}
\end{equation}

\end{theorem}

\begin{proof}

(i) Note that for any $(t_0,x_0)\in\bR^{d+1}$, 
\begin{align*}
Lf (t+t_0,x+x_0)&= \int_{-\infty}^{t+t_0}  \int_{\bR^d} q(t+t_0-s,x+x_0-y)f(s,y) dy ds
\\
&= \int_{-\infty}^{t}  \int_{\bR^d} q(t-s,x-y)f(s+t_0,x_0+y) dy ds
\\
&=L \left(f(\cdot+t_0,\cdot+x_0)\right)(t,x).
\end{align*}
Hence, by denoting $\tilde{f}(t,x):=f(t+t_0, x+x_0)$,
\begin{align*}
\aint_{Q_b(t_0,x_0)} |L f(t,x)-(L f)_{Q_b(t_0,x_0)}| dtdx
= \aint_{Q_b} |L \tilde{f}(t,x)-(L \tilde{f})_{Q_b}| dtdx.
\end{align*}
Moreover, since  $L_{\infty}$-norm is invariant under the translation,  for the proof of   (i)  it suffices to prove \eqref{eqn 09.28.15:28}, which we already proved when $f\in C_c^\infty (\bR^{d+1})$. 

Now let $f\in L_2(\bR^{d+1})\cap L_\infty(\bR^{d+1})$. We can take a sequence of functions $f_n\in C_c^\infty(\bR^{d+1})$ such that $L f_n \to L f \ (a.e.)$, and $\|f_n\|_{L_\infty(\bR^{d+1})}\leq \|f\|_{L_\infty(\bR^{d+1})}$. Then by Fatou's lemma,
\begin{align*}
\aint_{Q_b} |L f(t,x)-(L f)_{Q_b}| dtdx
& \leq \aint_{Q_b} \aint_{Q_b} |L f(t,x)-L f(s,y)| dtdxdsdy
\\
& \leq  \liminf_{n\to \infty} \aint_{Q_b} \aint_{Q_b} |L f_n(t,x)-L f_n(s,y)| dtdxdsdy
\\
& \leq C  \liminf_{n\to \infty} \|f_n\|_{L_\infty(\bR^{d+1})}
\leq C \|f\|_{L_\infty(\bR^{d+1})}.
\end{align*}
Hence, we have the first assertion.

(ii) 
\textbf{Step 1.} We prove \eqref{qpestimate} for the case $p=q$. By Lemma \ref{22estimate} and Theorem \ref{feffermanstein}, for any $f\in L_2(\bR^{d+1})\cap L_\infty(\bR^{d+1})$, it holds that $\|(L f)^{\#}\|_{L_2(\bR^{d+1})}\leq C \|f\|_{L_2(\bR^{d+1})}$.
By the first assertion, we have
$
\|(L f)^{\#}\|_{L_\infty(\bR^{d+1})}\leq C \|f\|_{L_\infty(\bR^{d+1})}$.
The linearity of $L$ implies the sublinearity of the map $f \to (L f)^{\#}$. Hence, by a version of the Marcinkiewicz interpolation theorem, for any $p\in [2,\infty)$ there exists a constant $C$ such that
\begin{equation*}
\|(L f)^{\#}\|_{L_p(\bR^{d+1})}\leq C \|f\|_{L_p(\bR^{d+1})}
\end{equation*}
for all $f\in L_2(\bR^{d+1})\cap L_\infty(\bR^{d+1})$. Finally, by Theorem \ref{feffermanstein}, we get
\begin{equation*}
\|L f\|_{L_p(\bR^{d+1})}\leq C \|f\|_{L_p(\bR^{d+1})}.
\end{equation*}
Therefore, we  have \eqref{qpestimate} for $p\in[2,\infty)$. For $p\in(1,2)$, use the standard duality argument.

\textbf{Step 2.} Now we prove (\ref{qpestimate}) for general $p,q\in(1,\infty)$.  Define $q(t,x):=0$ for $t\leq 0$. For each $(t,s)\in\bR^2$, we define the operator $G(t,s)$ as follows:
\begin{equation*}
G(t,s)f(x):=\int_{\bR^d} q(t-s,x-y)f(y) dy, \quad f\in C_c^\infty(\R^{d}).
\end{equation*}
Let $p\in(1,\infty)$. Then, by \eqref{eqn 09.23.16:51} with $k=1$,
\begin{align*}
\|G(t,s)f\|_{L_p(\R^{d})}&=\left\|\int_{\bR^d} q(t-s,x-y)f(y) dy\right\|_{L_p(\R^{d})}
\leq \|f\|_{L_p(\R^{d})} \int_{\bR^d} |q(t-s,y)|dy \leq  C |t-s|^{-1}\|f\|_{L_p(\R^{d})}.
\end{align*}
Hence, we can extend the operator $G(t,s)$  to $L_p(\R^{d})$ for $t\neq s$. Denote
\begin{equation*}
A:=[t_0,t_0+\delta), \quad A^*:=[t_0-\delta,t_0+2\delta), \quad  \delta>0.
\end{equation*}
Note that for $t\notin A^*$ and $s_1,s_2\in A$, we have
$|s_1-s_2|\leq\delta$, and $|t-(t_0+\delta)|\geq\delta$.
Thus for such $t,s_{1},s_{2}$ and for any $f\in L_p(\R^{d})$ with $\|f\|_{L_p(\R^{d})}=1$,
\begin{align*}
\|G(t,s_1)f-G(t,s_2)f\|_{L_p(\R^{d})}
&=\left\| \int_{\bR^d}  \left(q(t-s_1,x-y)-q(t-s_2,x-y)\right)f(y) dy\right\|_{L_p(\R^{d})}
\\
&\leq \|f\|_{L_p(\R^{d})} \int_{\bR^d} \left|q(t-s_1,x-y)-q(t-s_2,x-y)\right| dy dx
\\
&\leq C \int_{\bR^d} \int_0^1 |\partial_{t}q(t-us_1+(1-u)s_2,x)| |s_1-s_2| du dx
\\
&= C \int_{\bR^d} \int_0^1 |\mathcal{L}^{2}p_{d}(t-us_1+(1-u)s_2,x)| |s_1-s_2| du dx
\leq  C \frac{|s_1-s_2|}{(t-(t_0+\delta))^2}
\end{align*}
due to \eqref{eqn 09.23.16:51} with $k=2$ (recall that $G(t,s)=0$ for $t\leq s$). Hence, we have
\begin{align*}
\|G(t,s_1)-G(t,s_2)\|_{\Lambda} \leq  C \frac{|s_1-s_2|}{(t-(t_0+\delta))^2}.
\end{align*}
where $\|\cdot\|_{\Lambda}$ denotes the norm of a linear operator $\Lambda$ on $L_p(\R^{d})$. Therefore,
\begin{align*}
&\int_{\bR\setminus A^*} \|G(t,s_1)-G(t,s_2)\|_{\Lambda} dt \leq C \int_{\bR\setminus A^*}\frac{|s_1-s_2|}{(t-(t_0+\delta))^2} dt
\\
&\leq C|s_1-s_2|\int_{|t-(t_0+\delta)|\geq \delta}\frac{1}{(t-(t_0+\delta))^2} dt \leq C \delta \int_\delta^\infty t^{-2}dt \leq C.
\end{align*}
Furthermore, by following the proof of \cite[Theorem 1.1]{krylov2001caideron}, one can easily check that for almost every $t$ outside of the support of $f\in C_c^\infty(\bR;L_p(\R^{d}))$,
\begin{equation*}
L f(t,x)=\int_{-\infty}^\infty G(t,s)f(s,x)ds
\end{equation*}
where $L$ denotes the  extension  to  $L_p(\bR^{d+1})$ which is verified in Step 1. Hence, by the Banach space-valued version of the Calder\'on-Zygmund theorem (e.g. \cite[Theorem 4.1]{krylov2001caideron}), our assertion is proved for $1<q\leq p$. For $1<p<q<\infty$, again use the duality argument.
The theorem is proved.
\end{proof}

\medskip

\mysection{Proof of Theorem \ref{main theorem2}}

In this section, we will prove Theorem \ref{main theorem2}. Due to Lemma \ref{basicproperty}, we only need to prove case $\gamma=0$.

\vspace{1mm}

\noindent{\textbf{Step 1}} (Existence and estimation of solution). 
\\
First assume $f\in C_c^\infty(\bR^{d+1}_+)$, and let $u(t,x)$ be a function with representation \eqref{u=qf}. Using Remark \ref{Hvaluedconti} and the integrability of $p_{d}(t,x)$, we can easily check 
 $D^m_x u$, $\mathcal{L}D^m_x u$ $\in C([0,T];L_p)$, and thus $u\in C^{\infty}_p([0,T]\times\bR^d)$. Also, by Lemma \ref{u=qfsolution},  $u$ satisfies equation \eqref{mainequation2}.
 
Now we show estimation \eqref{mainestimate2} and \eqref{eqn 05.27.14:12}. Take $\eta_{k}=\eta_{k}(t)\in C^{\infty}(\bR)$ such that $0\leq \eta_{k}\leq 1$, $\eta_{k}(t)=1$ for $t\leq T+1/k$ and $\eta_{k}(t)=0$ for $t\geq T+2/k$. It is easy to see that $f\eta_{k}\in L_{q}(\bR;L_{p}(\bR^{d}))$, and $f(t)=f\eta_{k}(t)$ for $t\leq T$. Hence, by Theorem \ref{thm 05.14.18:28} (ii), we have
\begin{equation*}
\begin{aligned}
\|\mathcal{L}u\|_{\mathbb{L}_{q,p}(T)} &= \|L f \|_{\mathbb{L}_{q,p}(T)} = \|L(f\eta_{k})\|_{\mathbb{L}_{q,p}(T)} 
\leq \|L(f\eta_{k})\|_{L_{q}(\bR;L_{p}(\bR^{d}))}\leq C \|f\eta_{k}\|_{L_{q}(\bR;L_{p}(\bR^{d}))}.
\end{aligned}
\end{equation*}
Hence, letting $k\to\infty$, we have
$\|\mathcal{L}u\|_{\mathbb{L}_{q,p}(T)} \leq C \|f\|_{\mathbb{L}_{q,p}(T)}
$
by the dominated convergence theorem. Also, by the relation $\|p_{d}(t,\cdot)\|_{L_{1}(\bR^{d})} =1$ and Minkowski's inequality, we can easily check that
$
\|u\|_{\mathbb{L}_{q,p}(T)} \leq C(T) \|f\|_{\mathbb{L}_{q,p}(T)}
$.
Therefore, using the above inequalities and \eqref{eqn 03.25.15:03} we obtain estimation  \eqref{mainestimate2} and \eqref{eqn 05.27.14:12}.
For general $f$,  we take a sequence of functions   $f_{n}\in \Ccinf(\R^{d+1}_{+})$  such that $f_n \to f$ in $\bL_{q,p}(T)$. Let $u_{n}$  denote the solution with representation  \eqref{u=qf} with $f_{n}$ in place of $f$. Then \eqref{mainestimate2} applied to $u_m-u_n$ shows that $u_{n}$ is a Cauchy sequence in 
$\bH^{\psi,2}_{q,p,0}(T)$. By taking $u$ as the limit of $u_{n}$ in $\bH^{\psi,2}_{q,p,0}(T)$, we find that $u$ satisfies \eqref{mainequation2} and estimation  \eqref{mainestimate2} and \eqref{eqn 05.27.14:12} also holds for $u$. 

\noindent\textbf{Step 2} (Uniqueness of solution).  
\\
Let  $u\in \bH_{q,p,0}^{\psi,2}(T)$ be a solution to equation \eqref{mainequation2} with $f=0$. Take $u_{n}\in \Ccinf(\bR^{d+1}_{+})$ which converges to $u$ in $\mathbb{H}^{\psi,2}_{q,p}(T)$, and let $f_{n}:=\partial^{\alpha}_{t}u_{n}-\mathcal{L}u_{n}$. Then by Lemma \ref{u=qfsolution}, $u_{n}$ satisfies representation \eqref{u=qf} with $f_{n}$. Hence, by the argument in \textbf{Step 1}, we have
$
\|u_{n}\|_{\mathbb{H}^{\psi,2}_{q,p}(T)} \leq C(T)  \|f_{n}\|_{\mathbb{L}_{q,p}(T)}
$.
Since $f_{n}=\partial_{t}u_{n}-\mathcal{L}u_{n}$ converges to $0$ due to the choice of $u_{n}$, we conclude that $u=0$ in $\mathbb{H}^{\psi,2}_{q,p}(T)$. The theorem is proved.

\mysection{acknowledgements}
The authors are sincerely grateful to the anonymous referee for careful reading, valuable comments, and finding errors. The paper could be considerably improved by the referee's comments.

\appendix

\mysection{}\label{s:App}
We first give the proof of Lemma \ref{basicproperty}.
\begin{proof}[Proof of Lemma \ref{basicproperty}]
(i) Let $u_{n}\in\mathbb{H}^{\psi,\gamma+2}_{q,p,0}(T)$ converge to $u$ in $\mathbb{H}^{\psi,\gamma+2}_{q,p}(T)$ and for each $n$, let $u_{n,k}\in C^{\infty}_{p}([0,T]\times\bR^{d})$ be a defining sequence of $u_{n}$ such that $u_{n,k}(0,\cdot)=0$. Then for any given $\varepsilon>0$, we can choose $n$ and $k$ such that
$$
\|u - u_{n}\|_{\mathbb{H}^{\psi,\gamma+2}_{q,p}(T)} \leq \varepsilon/2, \quad \|u_{n} - u_{n,k}\|_{\mathbb{H}^{\psi,\gamma+2}_{q,p}(T)}\leq \varepsilon/2,
$$
and this certainly shows that $u\in \mathbb{H}^{\psi,\gamma+2}_{q,p,0}(T)$.

(ii) Due to the definition of $\mathbb{H}^{\psi,\gamma+2}_{q,p,0}(T)$, we only need to show that for given $u\in C^{\infty}_{p}([0,T]\times\bR^{d})$ with $u(0,\cdot)=0$, there exists a sequence of functions $u_{n}\in \Ccinf(\bR^{d+1}_{+})$ which converges to $u$ in $\mathbb{H}^{\psi,\gamma+2}_{q,p}(T)$. Moreover, using Remark \ref{Hvaluedconti} (ii) and considering a multiplication with a smooth cut-off function of $x$, we can further assume that $u$ has compact support,  that is, $u(t,x)=0$ whenever $|x|>R$ for some $R>0$.

Extend $u=0$ for $t\notin[0,T]$. Take a nonnegative smooth function $\eta_1\in C_c^\infty((1,2))$ so that
$\int_0^\infty \eta_1(t)dt=1$.
For $\varepsilon >0$, we define
$\eta_{1,\varepsilon}(t)=\varepsilon^{-1}\eta_1(t/\varepsilon)$, and
\begin{equation*}
u^{\varepsilon}(t,x)=\eta(t)\int_0^\infty u(s,x)\eta_{1,\varepsilon}(t-s)ds,
\end{equation*}
where $\eta\in C^\infty([0,\infty))$ such that $\eta(t)=1$ for all $t\leq T+1$ and vanishes for all large $t$. Then, since  $\eta_1\in C_c^\infty((1,2))$, 
\begin{align*}
u^{\varepsilon}(t,x)=0 \qquad \forall t<\varepsilon, \quad \forall x\in\bR^d,
\end{align*}
and $u^{\varepsilon}\in\Ccinf(\R^{d+1}_{+})$. Since $u(0,x)=0$, one can prove 
\begin{equation*}
\partial_t u^{\varepsilon}(t)=(\partial_t u)^{\varepsilon}(t), \quad t\leq T.
\end{equation*}
Therefore, for any $n\in \bN$,
$$
\|u^{\varepsilon}-u\|_{L_q([0,T];H^{2n}_p)}+\|\partial_t u^{\varepsilon}-\partial_t u \|_{L_q([0,T];H^{2n}_p)} \to 0
$$
as $\varepsilon \downarrow 0$.  This and Remark \ref{Hvaluedconti} (ii) implies that $u_n:=u^{1/n}$ converges to $u$ in $\mathbb{H}^{\psi,\gamma+2}_{q,p}(T)$.
Therefore, (ii) is proved.

(iii) Since $(1-\mathcal{L})^{\nu/2}$ is an isometry from $H^{\psi,\gamma+2}_{p}$ to $H^{\psi,\gamma-\nu+2}_{p}$, we can easily prove the desired result. The lemma is proved.

\end{proof}

Now we give some auxiliary results related to completely monotone functions. We recall the following definition;
\begin{definition}
We say that a function $f :(0,\infty) \to (0,\infty)$ is completely monotone if $(-1)^{n}f^{(n)} \geq 0$ for any $n \in \bN$.
\end{definition}
Also, recall that the following notation
$$
\mathcal{T}f(r):=-\frac{1}{r}f'(r).
$$
\begin{lemma}\label{lem 08.31.16:04}
Let $f$ be a completely monotone function. 

(i) For $n\in\bN$, define
$$
f_{1}(r):= \mathcal{T}f(r):=-\frac{1}{r}f'(r) ,\quad f_{n}(r):=\mathcal{T}f_{n-1}(r).
$$
Then $f_{n}$ is nonnegative decreasing function in $r\in(0,\infty)$.

(ii) Let $\nu(r):=r^{-d}f(r)$ and suppose 
\begin{equation}\label{eqn 08.26.15:25}
(-1)^{k}f^{(k)}(r) \leq C_{k}r^{-k}f(r) \quad \forall\, r>0 \quad k\in \bN,
\end{equation}
where the constant $C_{k}$ is independent of $r$. Then, we have
$$
r^{-n} \nu(r) \leq (-1)^{n}\nu^{(n)}(r) \leq C r^{-n}\nu(r) \quad \forall\, r\in(0,\infty), \quad \forall\, n\in\bN,
$$
where the constant $C$ depends only on $d,n,C_{1},\dots,C_{n}$.
\end{lemma}

\begin{proof}
(i) We first show that 
\begin{equation}\label{eqn 07.19.16:19}
f_{n}(r) = \sum_{k=0}^{n} C_{n,k}\left( -\frac{1}{r} \right)^{2n-k} f^{(k)}(r),
\end{equation}
where $C_{n,0}=0$, $C_{n,n}=1$ for $n\geq 1$, and $C_{n,k} = (2n-k) C_{n-1,k} + C_{n-1,k-1} >0$ for $n\geq 2$, $k\leq n-1$.
Due to the definition of $\mathcal{T}$, it is easy to check that $f_{1}$ satisfies \eqref{eqn 07.19.16:19}. Now suppose that $f_{n}$ satisfies \eqref{eqn 07.19.16:19}. Then direct computation yields
\begin{align*}
f_{n+1}(r) = -\frac{1}{r} \frac{d}{dr}f_{n}(r) =  \sum_{k=1}^{n} A_{n,k}  \frac{f^{(k)}(r)}{r^{2(n+1)-k}} + B_{n,k} \frac{f^{(k+1)}(r)}{r^{2(n+1)-(k+1)}},
\end{align*}
where $A_{n,k} = (-1)^{2(n+1)-k}(2n-k)C_{n,k}$, $B_{n,k} = (-1)^{2(n+1)-(k+1)}C_{n,k}$. Hence, due to the choice of $C_{n,k}\geq0$ above , it follows that
$$
f_{n+1}(r) = \sum_{k=0}^{n+1} C_{n+1,k} \left( - \frac{1}{r} \right)^{2(n+1)-k}f^{(k)}(r)
$$
Thus \eqref{eqn 07.19.16:19} also holds for $n+1$. By induction argument, \eqref{eqn 07.19.16:19} holds for all $n\in\bN$.

Since $f$ is completely monotone, we can check that $(-1)^{k}f^{(k)}$ is nonnegative and decreasing for all $k=1,2,\dots$. Thus for each $n\geq 1$,
\begin{align*}
f_{n}(r) = \sum_{k=0}^{n} C_{n,k}\left( -\frac{1}{r} \right)^{2n-k} f^{(k)}(r) & = \sum_{k=0}^{n} C_{n,k}\left(\frac{1}{r} \right)^{2n-k} (-1)^{2n-k}f^{(k)}(r)
 = \sum_{k=0}^{n} C_{n,k}\left(\frac{1}{r} \right)^{2n-k} (-1)^{k}f^{(k)}(r)
\end{align*}
is a nonnegative decreasing function. 

(ii) By the product rule of differentiation, we have
\begin{align*}
(-1)^{n} \nu^{(n)}(r) = \sum_{k=0}^{n} \binom{n}{k}(-1)^{n+k}D_{n,k}r^{-d-k}f^{(n-k)}(r),
\end{align*}
where $D_{n,0}=1$ $D_{n,1}=d$ and $D_{n,k}=d \times \dots \times (d+k-1)$ (for $2\leq k\leq n$). Using $(-1)^{n+k}f^{(n-k)}=(-1)^{n-k}f^{(n-k)}$ and complete monotonicity of $f$, we can easily see that $(-1)^{n}\nu^{(n)}(r) \geq r^{-n}\nu(r)$. Also, using the assumption on $f$, we have
\begin{align*}
(-1)^{n}\nu^{(n)}(r) &= \sum_{k=0}^{n} \binom{n}{k}(-1)^{n+k}D_{n,k}r^{-d-k}f^{(n-k)}(r)
\\
& = \sum_{k=0}^{n} \binom{n}{k}(-1)^{n-k}D_{n,k}r^{-d-k}f^{(n-k)}(r)
\\
& \leq \sum_{k=0}^{n} \binom{n}{k}D_{n,k}C_{k}r^{-d-n}f(r) = \left( \sum_{k=0}^{n} \binom{n}{k}D_{n,k}C_{k}\right) r^{-n}\nu(r).
\end{align*}
Therefore, by taking 
$C= \sum_{k=0}^{n} \binom{n}{k}D_{n,k}C_{k}$,
we have the desired result. The lemma is proved.
\end{proof}

\begin{lemma}\label{lem 08.31.16:04-2}
Let $\alpha \in (0,1)$.

(i) For each $k\in \bN$,
$$
g(r):=\left(  \frac{r^{-\alpha}}{1+r^{-\alpha}}  \right)^{k}
$$
is a competely monotone function satisfying \eqref{eqn 08.26.15:25}.

(ii) Let $f(r):= \log{(1+r^{-\alpha})}$. Then  $f$ is a completely monotone function satisfying \eqref{eqn 08.26.15:25}.
\end{lemma}

\begin{proof}
(i) It is known that $
\tilde{g}(r) =  r^{\alpha}/(1+r^{\alpha}) = 1/(1+r^{-\alpha}) 
$
is a Bernstein function (see \cite[Chapter 16]{SSV12}) since $\alpha\in(0,1)$. Thus 
$g^{1/k}(r) = 1-\tilde{g}(r) = r^{-\alpha}/(1+r^{-\alpha})
$
is completely monotone. Also, by the product rule of differentiation and the complete monotonicity of $g^{1/k}$,
\begin{align*}
(-1)^{n}(g^{2/k})^{(n)}(r) 
 =  \sum_{m=0}^{n} \binom{n}{m} (-1)^{m}(g^{1/k})^{(m)}(r) (-1)^{n-m}(g^{1/k})^{(n-m)}(r) \geq 0 \quad \forall \, n\in \bN.
\end{align*}
Hence, $g^{2/k}$ is completely monotone. Using induction, we can check that  $g$ is completely monotone. Also, one can check that for each $n\in \bN$,
$$
(-1)^{n}(g^{1/k})^{(n)}(r) = \frac{P_{n-1}(r^{-\alpha})}{(1+r^{-\alpha})^{n+1}}r^{-\alpha}r^{-n},
$$
where $P_{n-1}=P_{n-1}(r)$ is a polymonial of degree $n-1$. Moreover, since $g^{1/k}$ is completely monotone, we deduce that $P_{n-1}(r^{-\alpha})$ is nonnegative. Therefore, we have
\begin{align*}
(-1)^{n}(g^{1/k})^{(n)}(r) &= \frac{P_{n-1}(r^{-\alpha})}{(1+r^{-\alpha})^{n+1}}r^{-\alpha}r^{-n}
 \leq \frac{P_{n-1}(r^{-\alpha})}{(1+r^{-\alpha})^{n}}\frac{r^{-\alpha}}{(1+r^{-\alpha})}r^{-n}
\\
&\leq C(k,n) \frac{r^{-\alpha}}{(1+r^{-\alpha})}r^{-n} = C(k,n)g^{1/k}(r)r^{-n}.
\end{align*}
Using this and the product rule of differentiation, 
\begin{align*}
(-1)^{n}(g^{2/k})^{(n)}(r)
& = \sum_{m=0}^{n} \binom{n}{m} (-1)^{m}(g^{1/k})^{(m)}(r) (-1)^{n-m}(g^{1/k})^{(n-m)}(r)
\\
&\leq \sum_{m=0}^{n} \binom{n}{m} C(k,n,m) g^{1/k}(r)r^{-m} g^{1/k}(r) r^{-n+m} 
 \leq  C(n,k) g^{2/k}(r) r^{-n}.
\end{align*}
By induction, we can check that $g$ is a completely monotone function satisfying \eqref{eqn 08.26.15:25}.

(ii) Using the fact that $\log{(1+r^{-1})}$ is completely monotone function, and \cite[Theorem 3.7 (ii)]{SSV12} with Bernstein function $\phi(r)=r^{\alpha}$, we deduce that $f(r)$ is a completely monotone function. One can easily check that $(-1)f'(r)=\alpha/r(r^{\alpha}+1)$ and using the relation $(1+r)\log{(1+r^{-1})} \geq 1/2$ for all $r>0$, we have $(-1)f'(r) \geq r^{-1} f(r)$. Also, observe that
$$
(-1)f'(r) =  \alpha \frac{1}{r(r^{\alpha}+1)} = \alpha \frac{1}{r} \, \frac{r^{-\alpha}}{(1+r^{-\alpha})} := \alpha f_{1}(r) \, f_{2}(r),
$$
and $f_{1}$, and $f_{2}$ are completely monotone function satisfying \eqref{eqn 08.26.15:25} (for $f_{2}$ use the first assertion with $k=1$). Hence, using the product rule of differentiation and the first assertion, we have that for $n\geq 2$, 
\begin{align*}
(-1)^{n}f^{(n)}(r) &=\alpha (-1)^{n-1} \left( f_{1} \, f_{2} \right)^{(n-1)}(r)\nonumber \\
&=\alpha(-1)^{n-1} \sum_{m=0}^{n-1} \binom{n-1}{m} (-1)^{m}m! (-1)^{m}r^{-m-1} (-1)^{n-1-m}f^{(n-1-m)}_{2}(r)  \nonumber
\\
&\leq C \sum_{m=0}^{n-1} \binom{n-1}{m} m!  r^{-m-1}r^{-n+1+m}f_{2}(r) \leq C r^{-n} f(r).  
\end{align*}
Therefore, $f$ satisfies \eqref{eqn 08.26.15:25}. The lemma is proved.
\end{proof}
\begin{lemma}\label{l:qest_offd}
Let  $f:(0,\infty)\to(0,\infty)$ be a strictly increasing continuous function and $f^{-1}$ be its inverse. Suppose that there exist $c,\gamma>0$ such that $(f(R)/f(r))\le c(R/r)^\gamma$ for $0<r\le R<\infty$. Then, for any $k>0$, there exists $C>0$ such that for any $b>0$
\begin{align*}
\int_{(f(b^{-1}))^{-1}}^\infty s^{-1}f^{-1}(s^{-1})^k ds \le Cb^{-k}.
\end{align*}
\end{lemma}
\begin{proof}
See \cite[Lemma A.3]{KP21}.
\end{proof}

\begin{lemma}\label{l:qest_case1_pf} 
Suppose the function $\ell$ satisfies Assumption \ref{ell_con} (i). Then,
there exist $C_{1}=C_{1}(d,\text{\boldmath{$\kappa$}}_{3},\text{\boldmath{$\delta$}},\ell)$ and $C_{2}=C_{2}(d,\text{\boldmath{$\kappa$}}_{2},\text{\boldmath{$\delta$}},\ell)$such that for any $b>0$
\begin{gather}
\int_{(h(b))^{-1}}^\infty \int_{b\leq |y| \leq h^{-1}(s^{-1})} |D_{x}q(s,y)| dy ds \le C_{1}b^{-1},\label{q1upper_case1_pf}\\
\int_{(h(4b))^{-1}}^\infty \int_{|y|\leq 4b} |q(s,y)| dy ds \le C_{2}. \label{eqn 09.14.11:27}
\end{gather}
\end{lemma}

\begin{proof}
By Theorem \ref{cor 09.13.16:58}, and Theorem \ref{thm 09.13.17:01}
\begin{equation*}
\begin{aligned}
&\int_{(h(b))^{-1}}^\infty \int_{b\leq |y| \leq h^{-1}(s^{-1})} |D_{x}q(s,y)| dy ds
\\
&\leq C \int_{(h(b))^{-1}}^\infty \int_{b}^{h^{-1}(s^{-1})} K(\rho) \rho^{-2} e^{-C^{-1}sh(\rho)} d\rho  ds
\\
&\quad + C \int_{(h(b))^{-1}}^\infty \int_{b}^{h^{-1}(s^{-1})} \left( s^{-1} (h^{-1}(s^{-1}))^{-d-1}\mathbf{1}_{sh(\rho)\geq1}+ \frac{K(\rho)}{\rho^{d+1}} \mathbf{1}_{sh(\rho)\leq1}\right) \rho^{d-1} d\rho ds
\\
&\leq C \int_{(h(b))^{-1}}^\infty \int_{b}^{h^{-1}(s^{-1})} K(\rho) \rho^{-2} e^{-C^{-1}sh(\rho)} d\rho  ds
 + C \int_{(h(b))^{-1}}^\infty \int_{b}^{h^{-1}(s^{-1})} s^{-1} (h^{-1}(s^{-1}))^{-d-1} \rho^{d-1} d\rho ds,
\end{aligned}
\end{equation*}
where the last inequality holds since $sh(\rho)\leq 1$ is equivalent to $\rho\geq h^{-1}(s^{-1})$. By Fubini's theorem, we have
\begin{align}\label{eqn 05.30.18:42}
\int_{(h(b))^{-1}}^\infty \int_{b}^{h^{-1}(s^{-1})} K(\rho) \rho^{-2} e^{-C^{-1}sh(\rho)} d\rho  ds  
\leq C \int_{b}^{\infty}\int_{(h(\rho))^{-1}}^{\infty} K(\rho) \rho^{-2}  e^{-C^{-1}sh(\rho)} ds d\rho
\leq C \int_{b}^{\infty} \rho^{-2} d\rho
= C b^{-1}.
\end{align}
Also, by \eqref{l:qest_offd} with $f(r)=h(r^{-1})$, we have
\begin{align}\label{eqn 09.15.16:52}
\int_{(h(b))^{-1}}^\infty \int_{b}^{h^{-1}(s^{-1})} s^{-1} (h^{-1}(s^{-1}))^{-d-1}\rho^{d-1} d\rho ds
\leq C \int_{(h(b))^{-1}}^\infty  s^{-1} (h^{-1}(s^{-1}))^{-1}  ds \leq C b^{-1}.
\end{align}
Combining this with \eqref{eqn 05.30.18:42}, we have \eqref{q1upper_case1_pf}.

Now we prove \eqref{eqn 09.14.11:27}. Again by Theorem \ref{cor 09.13.16:58}, and Theorem \ref{thm 09.13.17:01}
\begin{align*}
&\int_{(h(4b))^{-1}}^\infty \int_{|y|\leq 4b} |q(s,y)| dy ds  \nonumber
\\
&\leq C \int_{(h(4b))^{-1}}^\infty \int_{|y| \leq 4b} \frac{K(|y|)}{|y|^{d}} e^{-C^{-1}sh(|y|)}  dy ds  \nonumber
\\
& \quad + C \int_{(h(4b))^{-1}}^\infty \int_{0}^{4b} \left( s^{-1}(h^{-1}(s^{-1}))^{-d}\mathbf{1}_{sh(\rho) \geq 1}  + \frac{K(\rho)}{\rho^{d}} \mathbf{1}_{sh(\rho) \leq 1}\right) \rho^{d-1}d\rho ds   \nonumber
\\
&\leq C \int_{(h(4b))^{-1}}^\infty \int_{0}^{4b} \rho^{-1} K(\rho) e^{-C^{-1}sh(\rho)}  d\rho ds  + C \int_{(h(4b))^{-1}}^\infty \int_{0}^{4b}  s^{-1}(h^{-1}(s^{-1}))^{-d} \rho^{d-1}d\rho ds,
\end{align*}
where the second inequality holds since $sh(\rho)\geq 1$ for $\rho\leq 4b$ and $s\geq (h(4b))^{-1}$. By Lemma \ref{rmk 05.07.16:12} (i), we can check that 
\begin{align}\label{eqn 09.15.14:24}
\int_{(h(4b))^{-1}}^\infty \int_{|y| \leq 4b} \frac{K(|y|)}{|y|^{d}} e^{-C^{-1}sh(|y|)}  dy ds 
&\leq C \int_{(h(4b))^{-1}}^\infty \int_{|y| \leq 4b} \frac{K(|y|)}{|y|^{d}} e^{-C^{-1}sh(|y|)/2} e^{-C^{-1}sh(4b)/2}  dy ds  \nonumber
\\
&\leq C \int_{(h(4b))^{-1}}^{\infty} e^{-C^{-1}sh(4b)/2} \int_{|y| \leq 4b} \frac{K(|y|)}{|y|^{d}}e^{-C^{-1}sh(|y|)/2}dy ds  \nonumber
\\
&\leq C  \int_{(h(4b))^{-1}}^{\infty}  s^{-1} e^{-C^{-1}sh(4b)/2} ds \leq C h(4b) \int_{(h(4b))^{-1}}^{\infty}  e^{-C^{-1}sh(4b)/2} ds \nonumber
\\
&\leq C h(4b)/h(4b) = C.
\end{align}
Also, due to Lemma \ref{l:qest_offd} with $f(r)=h(r^{-1})$, we have
\begin{align}\label{eqn 09.15.14:26}
 \int_{(h(4b))^{-1}}^\infty \int_{0}^{4b}  s^{-1}(h^{-1}(s^{-1}))^{-d} \rho^{d-1}d\rho ds  
\leq C b^{d} \int_{(h(4b))^{-1}}^\infty   s^{-1}(h^{-1}(s^{-1}))^{-d}ds \leq C.
\end{align}
Therefore, we have \eqref{eqn 09.14.11:27}, and the lemma is proved.
\end{proof}

\medskip

The following lemma is a counterpart of Lemma \ref{l:qest_case1_pf}. The proof is more delicate than that of Lemma \ref{l:qest_case1_pf} because $h(r)$ and $\ell(r^{-1})$ may not be comparable for $0<r\le1$. 

\begin{lemma}\label{lem 09.21.12:37}
Suppose the function $\ell$ satisfies Assumption \ref{ell_con} (ii)--(2). Then,
there exists $C_{1}=C_{1}(d,\text{\boldmath{$\kappa$}}_{3},\text{\boldmath{$\delta$}},\ell)$ and $C_{2}=C_{2}(d,\text{\boldmath{$\kappa$}}_{2},\text{\boldmath{$\delta$}},\ell)$ such that for any $b>0$
\begin{gather}
\int_{(h(b))^{-1}}^\infty \int_{b\leq |y| \leq h^{-1}(s^{-1})} |D_{x}q(s,y)| dy ds \le C_{1}b^{-1}, \label{q1upper_case2_pf}\\
\int_{(h(4b))^{-1}}^\infty \int_{|y|\leq 4b} |q(s,y)| dy ds \le C_{2}. \label{eqn 09.14.11:55}
\end{gather}
\end{lemma}

\begin{proof}
We first show \eqref{eqn 09.14.11:55}. By Theorem \ref{thm 09.13.17:01} with $t_{1,0}$,
\begin{align*}
\begin{split}
\int_{(h(4b))^{-1}}^{\infty} \int_{|y|\leq 4b}|q(s,y)| dy ds
&\leq \int_{(h(4b))^{-1}}^{\infty} \int_{|y|\leq 4b} \mathbf{1}_{s\leq t_{1,0}}\mathbf{1}_{s\leq a (\ell^*(|y|^{-1}))^{-1}}|q(s,y)| dy ds
\\
&\quad+\int_{(h(4b))^{-1}}^{\infty} \int_{|y|\leq 4b} \mathbf{1}_{s\leq t_{1,0}}\mathbf{1}_{s\geq a (\ell^*(|y|^{-1}))^{-1}} |q(s,y)| dy ds
\\
&\quad + C \int_{(h(4b))^{-1}}^\infty \int_{0}^{4b} s^{-1}(h^{-1}(s^{-1}))^{-d}\mathbf{1}_{sh(\rho) \geq 1} \rho^{d-1}d\rho ds 
\\
&=:I+II+III,
\end{split}
\end{align*}
where $a\geq \alpha_{1,0}$ comes from Theorem \ref{thm 08.24.16:27}. Since $s\leq a(\ell^{\ast}(|y|^{-1}))^{-1}$ is equivalent to $\theta_{a}(|y|,s) = |y|$, using Theorem \ref{thm 08.24.16:27} and \eqref{eqn 09.15.14:24}  we have $I\leq C$. 

Now we consider $II$. Observe that
\begin{align*}
II
&\leq \int_{(h(4b))^{-1}}^{\infty} \int_{|y|\leq 4b}\mathbf{1}_{s\leq t_{1,0}}\mathbf{1}_{a (\ell^*(|y|^{-1}))^{-1} \leq s \leq a(\ell^{\ast}((4b)^{-1}))^{-1}} |q(s,y)| dy ds
\\
& \quad + \int_{(h(4b))^{-1}}^{\infty} \int_{|y|\leq 4b} \mathbf{1}_{s\leq t_{1,0}} \mathbf{1}_{a(\ell^*((4b)^{-1}))^{-1} \leq s} |q(s,y)| dy ds
\\
&=: II_{1}+II_{2}.
\end{align*}
Since $r\mapsto h(r)$ is decreasing, we see that $h((\ell^{-1}(a/s))^{-1})\geq h(4b)$ for $s\le a(\ell^*((4b)^{-1}))^{-1}$. Using this and Theorem \ref{thm 08.24.16:27}, we have
\begin{align*}
II_{1}
&\leq \int_{(h(4b))^{-1}}^\infty \int_{|y| \leq 4b} \mathbf{1}_{a (\ell^*(|y|^{-1}))^{-1} \leq s \leq a(\ell^{\ast}((4b)^{-1}))^{-1}} \frac{K(\theta_{a}(|y|,s))}{\theta_{a}(|y|,s)^{d}} e^{-C^{-1}sh(\theta_{a}(|y|,s))}  dy ds
\\
&\leq \int_{(h(4b))^{-1}}^\infty \int_{|y| \leq 4b} e^{-C^{-1}sh(4b)/2} \frac{K(\theta_{a}(|y|,s))}{\theta_{a}(|y|,s)^{d}} e^{-C^{-1}sh(\theta_{a}(|y|,s))/2}  dy ds.
\end{align*}
Hence, by following the argument in \eqref{eqn 09.15.14:24} with Lemma \ref{rmk 05.07.16:12} (ii), we have $II_{1}\leq C$.

Recall that $\theta_{a}(|y|,s) = (\ell^{-1}(a/s))^{-1}$ on the region of the integral $II_{2}$. Therefore, we have
\begin{align*}
II_{2}
&\leq C \int_{(h(4b))^{-1}}^{\infty} \int_{|y|\leq 4b} \mathbf{1}_{s\leq t_{1,0}}\mathbf{1}_{a(\ell^{\ast}((4b)^{-1}))^{-1} \leq s} \frac{K(\theta_{a}(|y|,s))}{\theta_{a}(|y|,s)^{d}} e^{-C^{-1}sh(\theta_{a}(|y|,s))} dy ds 
\\
&= C \int_{(h(4b))^{-1}}^{t_{1,0}} \int_{|y|\leq 4b} \mathbf{1}_{a(\ell^{\ast}((4b)^{-1}))^{-1} \leq s} \frac{K((\ell^{-1}(a/s))^{-1})}{(\ell^{-1}(a/s))^{-d}} e^{-C^{-1} s h((\ell^{-1}(a/s))^{-1})} dy ds
\\
&\leq C \int_{(h(4b))^{-1}}^{t_{1,0}} \int_{|y|\leq 4b} b^{-d} h((\ell^{-1}(a/s))^{-1}) e^{-C^{-1}a \frac{h((\ell^{-1}(a/s))^{-1})}{\ell^{\ast}(\ell^{-1}(a/s))}} dy ds,
\end{align*}
where for the last inequality we used $s=a(\ell^{\ast}(\ell^{-1}(a/s)))^{-1}$. Hence, using \eqref{eqn 08.12.17:59} (recall Remark \ref{rmk 09.15.18:04} (i)) and the fact that $\ell^{\ast} \asymp \ell$, we have
\begin{align*}
II_{2}
\leq C \int_{(h(4b))^{-1}}^{t_{1,0}}  h((\ell^{-1}(a/s))^{-1}) e^{-C^{-1}a \frac{h((\ell^{-1}(a/s))^{-1})}{\ell(\ell^{-1}(a/s))}} ds
\leq C \int_{(h(4b))^{-1}}^{t_{1,0}}  1 ds \leq C,
\end{align*}
where for the last inequality, we abuse the notation $\int_{(h(4b))^{-1}}^{t_{1,0}} 1 ds =0$ for $t_{1,0}<(h(4b))^{-1}$. Thus, we obtain $II\le C$. Since we already handled $III$ in \eqref{eqn 09.15.14:26}, we have \eqref{eqn 09.14.11:55}.

Now, we show \eqref{q1upper_case2_pf}. 
Similar to \eqref{eqn 09.14.11:55}, we split the integral using Theorem \ref{thm 09.13.17:01} with $t_{1,1}$. Then we have a similar decomposition
\begin{align*}
\int_{(h(b))^{-1}}^\infty \int_{b\leq |y| \leq h^{-1}(s^{-1})} |D_{x}q(s,y)| dy ds
&\leq \int_{(h(b))^{-1}}^\infty \int_{b\leq |y| \leq h^{-1}(s^{-1})} \mathbf{1}_{s\leq t_{1,1}}\mathbf{1}_{s\leq a (\ell^*(|y|^{-1}))^{-1}}|D_{x}q(s,y)| dy ds
\\
&\quad + \int_{(h(b))^{-1}}^\infty \int_{b\leq |y| \leq h^{-1}(s^{-1})} \mathbf{1}_{s\leq t_{1,1}}\mathbf{1}_{s\geq a (\ell^*(|y|^{-1}))^{-1}}|D_{x}q(s,y)| dy ds
\\
&\quad + C \int_{(h(b))^{-1}}^\infty \int_{b}^{h^{-1}(s^{-1})} s^{-1} (h^{-1}(s^{-1}))^{-d-1}\mathbf{1}_{sh(\rho)\geq1} \rho^{d-1} d\rho ds
\\
& := IV+V+VI,
\end{align*}
where $a\geq \alpha_{1,1}$ comes from Theorem \ref{thm 08.24.16:27}. Note that the first equality holds since $sh(\rho)\leq1$ is equivalent to $\rho\geq h^{-1}(s^{-1})$. Recall that due to \eqref{eqn 09.15.16:52}, we have $VI\leq Cb^{-1}$. Therefore, we only focus on $IV$ and $V$.   

Since $s \leq a(\ell^{\ast}(|y|^{-1}))^{-1}$ is equivalent to $\theta_{a}(|y|,s)=|y|$, using Theorem \ref{thm 08.24.16:27} and following the argument in \eqref{eqn 05.30.18:42}, we have
\begin{align*}
IV \leq C \int_{(h(b))^{-1}}^\infty \int_{b}^{h^{-1}(s^{-1})} K(\rho) \rho^{-2} e^{-C^{-1}sh(\rho)} d\rho  ds \leq Cb^{-1}.
\end{align*}

Now we consider $V$. Note that $\theta_{a}(|y|,s)=(\ell^{-1}(a/s))^{-1} \geq |y|$ on the region of integral $V$. Using this and Theorem \ref{thm 08.24.16:27}, we have
\begin{align*}
V
&\leq C \int_{(h(b))^{-1}}^{t_{1,1}} \int_{b\leq |y| \leq h^{-1}(s^{-1})} \frac{K((\ell^{-1}(a/s))^{-1})}{|y|^{d+1}} e^{C^{-1}sh((\ell^{-1}(a/s))^{-1})}  dy ds
\\
&\leq C \int_{(h(b))^{-1}}^{t_{1,1}} \int_{b\leq |y| \leq h^{-1}(s^{-1})} \frac{h((\ell^{-1}(a/s))^{-1})}{|y|^{d+1}} e^{-C^{-1}a \frac{h((\ell^{-1}(a/s))^{-1})}{\ell^{\ast}(\ell^{-1}(a/s))}}  dy ds
\\
&\leq C \int_{(h(b))^{-1}}^{t_{1,1}} \int_{b\leq |y| \leq h^{-1}(s^{-1})} \frac{h((\ell^{-1}(a/s))^{-1})}{|y|^{d+1}} e^{-C^{-1}a \frac{h((\ell^{-1}(a/s))^{-1})}{\ell(\ell^{-1}(a/s))}}  dy ds.
\end{align*}
where we used $s=a(\ell^{\ast}(\ell^{-1}(a/s)))^{-1}$, and the fact that $\ell^{\ast} \asymp \ell$. Using \eqref{eqn 08.12.17:59} (recall Remark \ref{rmk 09.15.18:04} (i)) and Fubini's theorem we have
\begin{align*}
V &\leq C \int_{(h(b))^{-1}}^{t_{1,1}} \int_{b\leq |y| \leq h^{-1}(s^{-1})} \frac{h((\ell^{-1}(a/s))^{-1})}{|y|^{d+1}} e^{-C^{-1}a \frac{h((\ell^{-1}(a/s))^{-1})}{\ell(\ell^{-1}(a/s))}}  dy ds.
\\
&\leq C \int_{(h(b))^{-1}}^{t_{1,1}} \int_{b}^{h^{-1}(s^{-1})} \rho^{-2} d\rho ds
\leq C \int_{(h(b))^{-1}}^{t_{1,1}} \int_{b}^{\infty} \rho^{-2} d\rho ds
\leq Cb^{-1}.
\end{align*}
Here, we abuse the notation $\int_{(h(b))^{-1}}^{t_{1,1}} 1 ds =0$ for $t_{1,1}<(h(b))^{-1}$. Thus, we obtain \eqref{q1upper_case2_pf}. The lemma is proved.
\end{proof}

\end{document}